\documentclass[3p]{elsarticle}

\journal{J. of Graph Theory}

\usepackage{graphicx, amssymb, amsthm, amsmath, epsfig, latexsym}
\usepackage[T1]{fontenc}
\usepackage[latin1]{inputenc}
\usepackage[english]{babel}

\usepackage{amsmath,enumerate}
\usepackage{amssymb}
\usepackage{amscd}
\usepackage{tabularx}
\usepackage{tikz}
\usetikzlibrary{snakes}
\usepackage{pgf}
\usepackage{subfigure}

\usepackage{amscd}
\usepackage{latexsym}
\usepackage{tabularx}
\usepackage{pgf}
\usepackage{pgfcore}
\usepackage{pgfbaseimage}
\usepackage{pgfbaselayers}
\usepackage{pgfbasepatterns}
\usepackage{pgfbaseplot}
\usepackage{pgfbaseshapes}
\usepackage{pgfbasesnakes}
\usepackage{tikz}
\usetikzlibrary{patterns}

\newtheorem{theorem}{Theorem}
\newtheorem{corollary}[theorem]{Corollary}
\newtheorem{definition}[theorem]{Definition}
\newtheorem{proposition}[theorem]{Proposition}
\newtheorem{lemma}[theorem]{Lemma}
\newtheorem{conjecture}[theorem]{Conjecture}

\newtheorem{claim}{Claim}

\long\def\symbolfootnote[#1]#2{\begingroup%
\def\thefootnote{\fnsymbol{footnote}}\footnote[#1]{#2}\endgroup}

\newcommand{\join}{\bowtie} %symbol for the join operation
\newcommand{\Sym}{\ominus}
\newcommand{\ID}{\gamma^{\text{\tiny{ID}}}}
\newcommand{\EID}{\gamma^{\text{\tiny{EID}}}}
\tikzstyle{graphnode}=[draw,shape=circle,fill=black,draw=black,minimum size=0.5pt,inner sep=1.5pt]
\tikzstyle{edgecode}=[line width =3 pt]
\tikzstyle{edge}=[thin]

\begin{document}
\begin{frontmatter}
\title{Identifying codes in line graphs\tnoteref{t1}}
\tnotetext[t1]{This research is
    supported by the ANR Project IDEA {\scriptsize $\bullet$} {ANR-08-EMER-007},  2009-2011.}

\author[LaBRI]{Florent Foucaud}
\author[Grenoble]{Sylvain Gravier}
\author[LaBRI,Grenoble]{Reza Naserasr}
\author[Grenoble] {Aline Parreau}
\author[LaBRI]{Petru Valicov}
\address[LaBRI]{LaBRI - Universit\'e de Bordeaux - CNRS,
351 cours de la Lib\'eration,
33405 Talence cedex, France.}
\address[Grenoble]{Institut Fourier 100, rue des Maths, 
BP 74, 38402 St Martin d'H\`eres cedex, France. }

\begin{abstract}
An identifying code of a graph is a subset of its vertices such that every
vertex of the  graph is uniquely identified by the set of its neighbours within 
the code. We study the edge-identifying code problem, i.e. the identifying
code problem in line graphs. If $\ID(G)$ denotes the size of a
minimum identifying code of an identifiable graph $G$, we show that the usual
bound $\ID(G)\ge \lceil\log_2(n+1)\rceil$, where $n$ denotes the order
of $G$, can be improved to $\Theta(\sqrt{n})$ in the class of line
graphs. Moreover, this bound is tight. We also prove that the upper
bound $\ID(\mathcal{L}(G))\leq 2|V(G)|-5$,
 where $\mathcal{L}(G)$ is the line graph of $G$, holds (with two exceptions). This implies
that a conjecture of R.~Klasing, A.~Kosowski, A.~Raspaud and the
first author holds for a subclass of line graphs. Finally, we show
that the edge-identifying code problem is NP-complete, even for the
class of planar bipartite graphs of maximum degree~3 and arbitrarily
large girth.
\end{abstract}


\end{frontmatter}

\section{Introduction}

An \emph{identifying code} of a graph $G$ is a subset $\mathcal C$ of vertices
of $G$ such that for each vertex $x$, the set of vertices in $\mathcal C$ at
distance at most~1 from $x$, is nonempty and uniquely identifies $x$. More formally:

\begin{definition}
Given a graph $G$, a subset $\mathcal C$ of $V(G)$ is an identifying code of $G$ if 
$\mathcal C$ is both:
\begin{itemize}
\item a \emph{dominating set} of $G$, i.e. for each vertex $v\in V(G)$, $N[v]\cap
\mathcal C\neq\emptyset$, and
\item a \emph{separating set} of $G$, i.e. for each pair $u, v \in V(G)$ ($u\neq v$), $N[u]\cap \mathcal C\neq
N[v]\cap \mathcal C$.
\end{itemize}
\end{definition}

Here $N[v]$ is the closed neighbourhood of $v$ in
$G$. This concept was introduced in 1998 in~\cite{KCL98} and is a
well-studied one (see
e.g.~\cite{A10,CHL03,CHL07,FGKNPV10,FKKR10,GM07,M06}).

A vertex $x$ is a \emph{twin} of another vertex $y$ if $N[x]=N[y]$.
A graph $G$ is called \emph{twin-free} if no vertex has a twin.
The first observation regarding the concept of identifying codes is that a graph is
\emph{identifiable}
if and only if it is twin-free. As usual for many other graph theory concepts, a natural problem
in the study of identifying codes is to find one of a minimum size. Given a graph $G$, the smallest
size of an identifying code of $G$ is called \emph{identifying code number} of $G$ and denoted by $\ID(G)$. The main lines of research here
are to find the exact value of $\ID(G)$ for interesting graph classes, to approximate it and to
give lower or upper bounds in terms of simpler graph parameters. Examples of classic results are
as follows:

\begin{theorem}\label{boundn-1}\cite{GM07}
If $G$ is a twin-free graph with at least two edges, then $\ID(G)\leq |V(G)|-1$. 
\end{theorem}

The collection of all twin-free graphs reaching this bound is classified in \cite{FGKNPV10}.

A better upper bound in terms of both number of vertices and maximum degree $\Delta(G)$ of a graph $G$
is also conjectured:

\begin{conjecture}\label{ConjFKKR}\cite{FKKR10}
There exists a constant  $c$ such that for every twin-free graph $G$,
$$\ID(G) \leq |V(G)|- \frac{|V(G)|}{\Delta(G)}+c.$$
\end{conjecture}

Some support for this conjecture is provided in~\cite{FGKNPV10, FKKR10, FP11}.

The parameter $\ID(G)$ is also bounded below by a function of $|V(G)|$ where equality
holds for infinitely many graphs.

\begin{theorem}\label{boundlogn}\cite{KCL98}
For any twin-free graph $G$, $\ID(G) \geq \lceil \log_2(|V(G)|+1) \rceil$.
\end{theorem}

The collection of all graphs attaining this lower bound is classified in~\cite{M06}.

From a computational point of view, it is shown that given a graph
$G$, finding the exact value of $\ID(G)$ is in the class of NP-hard
problems. It in fact remains NP-hard for many subclasses of graphs
\cite{A10,CHL03}.  Furthermore, approximating $\ID(G)$ is not easy
either as shown in \cite{BLT06,GKM08,S07}: it is NP-hard to
approximate $\ID(G)$ within a $o(\log(|V(G)|))$-factor.

The problem of finding identifying codes in graphs can be viewed as a special case 
of the more general combinatorial problem of finding \emph{transversals} in hypergraphs
(a transversal is a set of vertices intersecting each hyperedge).
More precisely, to each graph $G$ one can associate the hypergraph $\mathcal H (G)$
whose vertices are vertices of $G$ and whose hyperedges are all the sets of the form $N[v]$
and $N[u]\Sym  N[v]$ (symmetric difference of $N[u]$ and $N[v]$).
Finding an identifying code for $G$ is then equivalent to finding a
transversal for $\mathcal H (G)$. Though the identifying code problem is captured by this
more general problem, the structural properties of the graph from which the hypergraph is
built allow one to obtain stronger results which are not true for general hypergraphs. In this
work, we show that even stronger results can be obtained if we consider hypergraphs coming
from line graphs. These stronger results follow from the new perspective of identifying edges
by edges.

Given a graph $G$ and an edge $e$ of $G$, we define $N[e]$ to be the set of edges adjacent
to $e$ together with $e$ itself. An \emph{edge-identifying code} of a graph $G$ is a subset $\mathcal{C}_E$
of edges such that for each edge $e$ the set $N[e] \cap \mathcal{C}_E$ is nonempty and uniquely
determines $e$. More formally:

\begin{definition}
Given a graph $G$, a subset $\mathcal{C}_E$ of $E(G)$ is an edge-identifying code of $G$ if 
$\mathcal{C}_E$ is both:
\begin{itemize}
\item an \emph{edge-dominating set} of $G$, i.e. for each edge $e\in E(G)$, $N[e]\cap
\mathcal{C}_E\neq\emptyset$, and
\item an \emph{edge-separating set} of $G$, i.e. for each pair $e, f \in E(G)$ ($e\neq f$), $N[e]\cap \mathcal{C}_E\neq
N[f]\cap \mathcal{C}_E$.
\end{itemize}
\end{definition}

We will say that an edge $e$ separates edges $f$ and $g$ if either $e$
belongs to $N[f]$ but not to $N[g]$, or vice-versa.  When considering
edge-identifying codes we will assume the edge set of the graph is
nonempty. The line graph $\mathcal L(G)$ of a graph $G$ is the graph
with vertex set $E(G)$, where two vertices of $\mathcal L(G)$ are
adjacent if the corresponding edges are adjacent in $G$. It is easily
observed that the notion of edge-identifying code of $G$ is equivalent
to the notion of (vertex-)identifying code of the \emph{line graph} of
$G$. Thus a graph $G$ admits an edge-identifying code if and only if $
\mathcal L(G)$ is twin-free.  A pair of twins in $\mathcal L(G)$ can
correspond in $G$ to a pair of: 1. parallel edges; 2. adjacent edges
whose non-common ends are of degree~1; 3. adjacent edges whose non
common ends are of degree~2 but they are connected to each
other. Hence we will consider simple graphs only. A pair of edges of type~2 or type~3
is called \emph{pendant} (see Figure~\ref{fig:pendant}) and thus a graph is \emph{edge-identifiable} if
and only if it is \emph{pendant-free}.  The smallest size of an
edge-identifying code of an edge-identifiable graph $G$ is denoted by
$\EID(G)$ and is called \emph{edge-identifying code number} of $G$.

\begin{figure}[ht!]
\centering
\subfigure{\scalebox{1.0}{\begin{tikzpicture}[join=bevel,inner sep=0.5mm, scale=0.8]
\node at (0,0) [draw,shape=circle,fill] (a) {};
\node at (-1,2) [draw,shape=circle,fill] (b) {};
\node at (1,2) [draw,shape=circle,fill] (c) {};
\node at (0,-1) {$G$};
\draw[line width=2pt] (a) -- (b)
      (a) -- (c);

\draw (b) -- (c);
\draw (0,-1) ellipse (1.5cm and 1cm);

\end{tikzpicture}}}\qquad
\subfigure{\scalebox{1.0}{\begin{tikzpicture}[join=bevel,inner sep=0.5mm, scale=0.8]
\path (0,0) node[draw,shape=circle,fill] (a) {};
\path (-1,2) node[draw,shape=circle,fill] (b) {};
\path (1,2) node[draw,shape=circle,fill] (c) {};
\path (0,-1) node {$G$};
\draw[line width=2pt] (a) -- (b)
      (a) -- (c);
\draw (0,-1) ellipse (1.5cm and 1cm);
\end{tikzpicture}}}
\label{fig:pendant}
\caption{Two possibilities for a pair of pendant edges (thick edges) in $G$}
\end{figure}
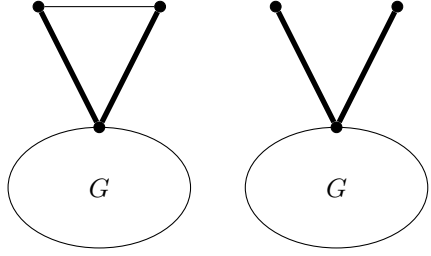

As we will use it often throughout the paper, given a graph $G$ and a
set $\mathcal{S}_E$ of its edges, we define the graph induced by $\mathcal{S}_E$
to be the graph with  the set of all endpoints of the edges
of $\mathcal{S}_E$ as its vertex set and $\mathcal{S}_E$ as its edge set.

To warm up, we notice that five edges of a perfect matching of the Petersen graph $P$,
form an edge-identifying code of this graph (see Figure~\ref{fig:petersen}). The lower
bound of Theorem~\ref{boundlogn} proves that $\EID(P)\geq 4$. Later, by improving this bound
for line graphs, we will see that in fact $\EID(P)=5$ (see Theorem~\ref{prop:low2}
and Theorem~\ref{Thm:Low1}).

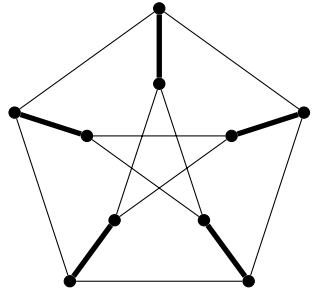
\begin{figure}[!ht]

\begin{center}
\begin{tikzpicture}[join=bevel,inner sep=0.5mm]
\node[graphnode](45) at (18:1) {};
\node[graphnode](15) at (90:1) {};
\node[graphnode](12) at (162:1) {};
\node[graphnode](23) at (234:1) {};
\node[graphnode](34) at (306:1) {};
\node[graphnode](13) at (18:2) {};
\node[graphnode](24) at (90:2) {};
\node[graphnode](35) at (162:2) {};
\node[graphnode](14) at (234:2) {};
\node[graphnode](25) at (306:2) {};

\draw[-] (12)--(34)--(15)--(23)--(45)--(12);
\draw[-] (13)--(24)--(35)--(14)--(25)--(13);
\draw[line width=2pt] (12)--(35);
\draw[line width=2pt] (15)--(24);
\draw[line width=2pt] (34)--(25);
\draw[line width=2pt] (23)--(14);
\draw[line width=2pt] (45)--(13);
\end{tikzpicture}
\end{center}
\caption{\label{fig:petersen} An edge-identifying code of the Petersen graph}
\end{figure}

The outline of the paper is as follows: in
Section~\ref{Preliminaries}, we introduce some useful lemmas and give
the edge-identifying code number of some basic families of graphs.  In
Section~\ref{LowerBound}, we improve the general lower bound for the
class of line graphs, then in Section~\ref{UpperBound} we improve the
upper bound. Finally, in Section~\ref{Complexity} we show that
determining $\EID(G)$ is also in the class of NP-hard problems even
when restricted to planar subcubic bipartite graphs of arbitrarily
large girth, but the problem is 4-approximable in polynomial time.

\section{Preliminaries}\label{Preliminaries}

In this section we first give some easy tools which help for finding
minimum-size edge-identifying codes of graphs.  We then apply these
tools to determine the exact values of $\EID$ for some basic families
of graphs.  We recall that $C_n$ is the cycle on $n$ vertices, $P_n$
is the path on $n$ vertices, $K_n$ is the complete graph on $n$
vertices and $K_{n,m}$ is the complete bipartite graph with parts of
size $n$ and $m$. We recall that the \emph{girth} of a graph is the
length of one of its shortest cycles. An \emph{edge cover} of a graph
$G$ is a subset $\mathcal{S}_E$ of its edges such that the union of
the endpoints of $\mathcal{S}_E$ equals $V(G)$.  A \emph{matching} is
a set of pairwise non-adjacent edges, and a \emph{perfect matching} is
a matching which is also an edge cover.

\begin{lemma}\label{lem:girth5}
Let $G$ be a simple graph with girth at least $5$. Let $\mathcal{C}_E$
be an edge cover of $G$ such that the graph $(V(G), \mathcal{C}_E)$ is
pendant-free. Then $\mathcal{C}_E$ is an edge-identifying code of
$G$. In particular, if $G$ has a perfect matching $M$, $M$ is an
edge-identifying code of $G$.
\end{lemma}

\begin{proof}
The code $\mathcal{C}_E$ is an edge-dominating set of $G$ because it
covers all the vertices of $G$. To complete the proof, we need to prove
that $\mathcal{C}_E$ is also an edge-separating set.  Let $e_1$, $e_2$
be two edges of $G$. If $e_1,e_2 \in \mathcal{C}_E$, then
$\mathcal{C}_E\cap N[e_1]\neq \mathcal{C}_E\cap N[e_2]$ because
$(V(G), \mathcal{C}_E)$ is pendant-free.  Otherwise, we can assume
that $e_2\notin \mathcal{C}_E$.  If $e_1\in \mathcal{C}_E$ and
$\mathcal{C}_E\cap N[e_1]=\mathcal{C}_E\cap N[e_2]$, then $e_2$ must
be adjacent to $e_1$.  Let $u$ be their common vertex and
$e_2=uv$. Since $\mathcal{C}_E$ is an edge cover, there is an edge
$e_3 \in \mathcal{C}_E$ which is incident to $v$. However, $e_3$
cannot be adjacent to $e_1$ because $G$ is triangle-free. Therefore
$e_3$ separates $e_1$ and $e_2$.  Finally, we assume neither of $e_1$
and $e_2$ is in $\mathcal{C}_E$. Then there are two edges of
$\mathcal{C}_E$, say $e_3$ and $e_4$, adjacent to the two ends of
$e_1$. But since $G$ has neither $C_3$ nor $C_4$ as a subgraph, $e_3$
and $e_4$ cannot both be adjacent to $e_2$ and, therefore, $e_1$ and
$e_2$ are separated.
\end{proof}

We note that in the previous proof the absence of $C_4$ is only used when the endpoints of
$e_1, e_2, e_3, e_4$ could induce a $C_4$ which would not be adjacent to any other edge of $\mathcal{C}_E$.
Thus, we have the following stronger statement: 

\begin{lemma}\label{lem:girth4}
Let $G$ be a triangle-free graph. Let $\mathcal{C}_E$ be a subset of
edges of $G$ that covers vertices of $G$, such that $\mathcal{C}_E$ is
pendant-free.  If for no pair $xy$, $uv$ of isolated edges in
$\mathcal{C}_E$, the set $\{x,y,u,v\}$ induces a $C_4$ in $G$, then
$\mathcal{C}_E$ is an edge-identifying code of $G$.
%If no pair of isolated edges in $\mathcal{C}_E$ induces a $C_4$, then $\mathcal{C}_E$ is an edge-identifying code of $G$.
\end{lemma}

We will also need the following lemma about pendant-free trees.

\begin{lemma}\label{degree2inPendantFree}
If $T$ is a pendant-free tree on more than two vertices, then $T$ has two vertices of degree~1,
each adjacent to a vertex of degree~2.
\end{lemma}

\begin{proof}
Take a longest path in $T$, then it is easy to verify that both ends of this path satisfy
the condition of the lemma.
\end{proof}

We are now ready to determine the value of $\EID$ of some families  of graphs.

\begin{proposition} 
We have 
$ \EID(K_n)=  
  \begin{cases} 
  5,  & \text{ if } n=4 \text{ or } 5\\
  n-1,  & \text{ if } n\geq 6 \\
    \end{cases}$. 
Furthermore, let $\mathcal{C}_E$ be an edge-identifying code of $K_n$ of size
$n-1$ ($n\geq 6$) and let $G_1, G_2, \ldots, G_k$ be the connected
components of $(V(K_n), \mathcal{C}_E)$.  Then exactly one component, say $G_i$,
is isomorphic to $K_1$ and every other component $G_j$ ($j\neq i$) is
isomorphic to a cycle of length at least~5.
\end{proposition}

\begin{proof}
We note that $\mathcal L(K_4)$ is isomorphic to $K_6\setminus M$, where $M$ is a perfect matching of $K_6$. One can check that this graph has identifying code number~$5$.
By a case analysis, we can show that $K_5$ does not admit an edge-identifying code of
size~4. Indeed, since an edge-identifying code must be pendant-free, there are only two graphs possible for an edge-identifying code of this size: a path $P_5$ or a cycle $C_4$. In both cases, there are edges which are not separated. %A 
Edges of a $C_5$ form an edge-identifying code of size~5 of $K_5$, hence $\EID(K_5)=5$. Furthermore,
it is not difficult to check that the set of edges of a cycle of length $n-1$ ($n\geq 6$) identifies
all edges of $K_n$. Thus we have $\EID(K_n) \leq n-1$. The fact that $\EID(K_n) \geq n-1$
follows from the second part of the theorem which is proved as follows.

Let $\mathcal{C}_E$ be an edge-identifying code of $K_n$ of size $n-1$ or less ($n\geq 6$).
Let $G'=(V(K_n), \mathcal{C}_E)$. Let $G_1, G_2, \ldots, G_k$ be the connected components
of $G'$. Since $G'$ has $n$ vertices but at most $n-1$ edges, at least one component of $G'$
is a tree. On the other hand we claim that at most one of these components can be a tree and that such
tree would be isomorphic to $K_1$. Let $G_i$ be a tree. First we show that $|V(G_i)|\leq 2$.
If not, by Lemma~\ref{degree2inPendantFree} there is a vertex $x$ of degree~1
in $G_i$ with a neighbour $u$ of degree~2. Let $v$ be the other neighbour of $u$.
Then the edges $xv$ and $uv$ are not identified. If $V(G_i)=\{ x, y \}$ then for any
other vertex $u$, the edges $ux$ and $uy$ are not separated. Finally, if there are
$G_i$ and $G_j$ with $V(G_i)=\{ x\}$ and $V(G_j)=\{ y\}$,  then the edge $xy$ is
not dominated by $\mathcal{C}_E$. Thus exactly one component of $G'$, say $G_1$, is a tree and
$G_1\cong K_1$. This implies that $\EID(K_n)\geq n-1$. Therefore, $\EID(K_n)=n-1$
and, furthermore, each $G_i$, $(i \geq 2)$, is a graph with a unique cycle.

It remains to prove that each $G_i$, $i\geq 2$ is isomorphic to a cycle of length
at least $5$. By contradiction suppose one of these graphs, say $G_2$,
is not isomorphic to a cycle. Since $G_2$ has a unique cycle, it must contain a vertex $v$ of degree~1. Let $t$
be the neighbour of $v$ in $G_2$ and let $u$ be the vertex of $G_1$. Then the edges
$tv$ and $tu$ are not separated by $\mathcal{C}_E$. Finally we note that such cycle cannot be of
length~3 or~4, because $C_3$ is not pendant-free and in $C_4$, the two chords (which
are edges of $K_n$) would not be separated.
\end{proof}

\begin{proposition}
  $\EID(K_{n,n})=\left\lceil \frac{3n-1}{2}\right\rceil$  for $n\geq 3$.
\end{proposition}

\begin{proof}
Let $X$ and $Y$ be the two parts of $K_{n,n}$. If $n$ is even, then let
$\displaystyle \{A_i\}_{i=1} ^{\frac{n}{2}}$
be a partition of vertices such that each $A_i$ has exactly two vertices in $X$ and
two in $Y$. Let $G_i$ be a subgraph of $K_{n,n}$ isomorphic to $P_4$ and with
$A_i$ as its vertices. 
If $n$ is odd, let $A_1$ be of size~2 and having exactly one element from $X$ and
one element from $Y$ and let also $G_1$ be the subgraph (isomorphic to $K_2$) induced
by $A_1$. Then we define $A_i$'s and $G_i$'s ($i\geq 2$) as in the previous case (for
$K_{n-1,n-1}$). By Lemma~\ref{lem:girth4}, the set of edges in the $G_i$'s induces an
edge-identifying code of $K_{n,n}$ of size $\left\lceil \frac{3n-1}{2}\right\rceil$. To complete
the proof we show that there cannot be any smaller edge-identifying code.

Let $\mathcal{C}_E$ be an edge-identifying code of $G$. Let $G_1, G_2, \ldots,
G_k$ be the connected components of $(X\cup Y, \mathcal{C}_E)$. The proof will
be completed if we show that except possibly one, every $G_i$ must
have at least four vertices. To prove this claim we first note that
there is no connected pendant-free graph on three vertices. We now
suppose $G_1$ and $G_2$ are both of order~2. Then the two edges
connecting $G_1$ and $G_2$ are not separated. If $G_1$ is of order~1
and $G_2$ is of order~2, then the edge connecting $G_1$ to $G_2$ is
not identified from the edge of $G_2$. If $G_1$ and $G_2$ are both of
order~1, then both of their vertices must be in the same part of the
graph as otherwise the edge connecting them is not dominated by
$\mathcal{C}_E$. But now for any vertex $x$ which is not in the same part as
$G_1$ and $G_2$, the edges connecting $x$ to $G_1$ and $G_2$ are not
separated.
\end{proof}

The following examples show that if true, the upper bound of Conjecture~\ref{ConjFKKR} is tight even in the class of line graphs. These examples were first introduced in \cite{FKKR10} but without using the notion
of edge-identifying codes.

\begin{proposition}\label{LineGraphofRegulars}
Let $G$ be a $k$-regular multigraph 
($k\geq 3$). Let $G_1$ be obtained from $G$
by subdividing each edge exactly once. Then
$\EID(G_1)=(k-1)|V(G)|=|E(G_1)|-\frac{|E(G_1)|}{2k-2}=
|V(\mathcal L(G_1))|-\frac{|V(\mathcal L(G_1))|}{\Delta(\mathcal L(G_1))}$.
\end{proposition}

\begin{proof}
Let $x$ be a vertex of $G_1$ of degree at least~3 (an original vertex from $G$). For
each edge $e_i^x$ incident to $x$, let ${e'_i}^x$ be the edge adjacent to $e_i^x$
but not incident to $x$ and let $A_x=\{{ e'^x_i}\}_{i=1}^k$. Then
$\{A_x ~ | ~x\in V(G) \}$ is a partition of $E(G_1)$. For any edge-identifying code $\mathcal{C}_E$ of $G_1$, if two elements of $A_x$,
say ${e'^x_1}$ and ${e'^x_2}$,  are both not in $\mathcal{C}_E$, then
$e^x_1$ and $e^x_2$ are not separated. Thus $|\mathcal{C}_E \cap A_x| \geq k-1$. This proves 
that $|\mathcal{C}_E| \geq (k-1)|V(G)|$. 

We now build an edge-identifying code of this size by choosing one
edge of each set $A_x$, in such a way that for each vertex $x$
originally from $G$, exactly one edge incident to $x$ is chosen. Then
the set of non-chosen edges will be an edge-identifying code. To
select this set of edges, one can consider the incident bipartite
multigraph $H$ of $G$: the vertex set of $H$ is $V\cup V'$ where $V$
and $V'$ are copies of $V(G)$ and there is an edge $xx'$ in $H$ if
$x\in V$, $x'\in V'$ and $xx'\in E(G)$. The multigraph $H$ is $k$-regular
and bipartite, thus it has a perfect matching $M$. For each vertex $x\in
V$, let $\rho(x)$ be the vertex in $V'$ such that $x\rho(x)\in M$. Let
now $e'^x_M$ be the edge of $G_1$ that belongs to the set $A_x$ and is
incident to $\rho(x)$ (in $G_1$). Finally, let
$\mathcal{C}_E=E(G_1)\setminus \{e'^x_M\}_{x\in V(G)}$. Exactly one
element of each $A_x$ is not in $\mathcal{C}_E$, and for each vertex
$x$, exactly one edge incident to $x$ is not in $\mathcal{C}_E$. This
implies that $\mathcal{C}_E$ is an edge-identifying code.
\end{proof}

For a simple example of the previous construction, let $G$ be the
multigraph on two vertices with $k$ parallel edges. Then $G_1\cong
K_{2,n}$ and therefore $\EID(K_{2,n})=2n-2$.

Hypercubes, being the natural ground of code-like structures, have been a center of focus for
determining the smallest size of their identifying codes. The hypercube of dimension $d$, denoted $\mathcal H_d$, is a graph whose vertices are elements
of $\mathbb Z_2^d$ with two vertices being adjacent if their difference is in the 
standard basis: $\{(1,0,0,\ldots, 0), (0,1,0,\ldots,0), \ldots, (0,0,\ldots,0,1) \}$.
The hypercube of dimension $d$ can also be viewed as the cartesian product of the hypercube 
of dimension $d-1$ and $K_2$. In this way of building $\mathcal H_d$ we add a new coordinate
to the left of the vectors representing the vertices of $\mathcal H_{d-1}$.

The problem of determining the identifying code number of hypercubes has proved to be a
challenging one from both theoretical and computational points of view. Today the precise identifying code number is known for only seven hypercubes \cite{CCHL10}.
In contrast, we show here that finding the edge-identifying code number of a hypercube is 
not so difficult. We first introduce the following general theorem.

\begin{theorem}\label{prop:low2}
Let $G$ be a connected  pendant-free graph. We have:
$$\EID(G)\geq \frac{|V(G)|}{2}.$$
\end{theorem}
\begin{proof}
Let $\mathcal{C}_E$ be an edge-identifying code of $G$. Let $G'$ be the subgraph induced by
$\mathcal{C}_E$ and let $G_1,\ldots,G_s$ be the connected components of $G'$.
Let $n_i$ be the order of $G_i$ and $k_i$ be its size (thus $\sum_{i=1}^{s}k_i=|\mathcal{C}_E|$).
Let $X=V(G)\setminus V(G')$ and $n_i'$ be the number of vertices in $X$ that are joined
to a vertex of $G_i$ in $G$. We show that $n'_i+n_i\leq 2k_i$. If $k_i=1$, then clearly
$n'_i=0$ and $n'_i+n_i=2=2k_i$. If $G_i$ is a tree, then $n_i=k_i+1$ and, by Lemma~\ref{degree2inPendantFree},
$G_i$ must have two vertices of degree~$2$ each having a vertex of degree~1 as a neighbour.
Then no vertex of $X$ can be adjacent to one of these two vertices in $G$. Moreover,
each other vertex of $G_i$ can be adjacent to at most one vertex in $X$. So
$n'_i\leq k_i-1$, and finally $n_i+n'_i\leq 2k_i$.
If $G_i$ is not at tree, we have $n_i\leq k_i$ and $n'_i\leq n_i$
and, therefore, $n'_i+n_i\leq 2k_i$.
Finally, since $G$ is connected, each vertex in $X$ is connected to at least one
$G_i$. Hence by counting the number vertices of $G$ we have:
$$|V(G)|\leq \sum_{i=1}^{s}(n_i+n'_i) \leq 2\sum_{i=1}^s k_i \leq
2|\mathcal{C}_E|.$$
\end{proof}

Theorem~\ref{prop:low2} together with Lemma~\ref{lem:girth4} leads to the
following result:

\begin{corollary}\label{cor:PMgirth4}
Let $G$ be a triangle-free pendant-free graph. Suppose $G$ has a
perfect matching $M$ with the property that for any pair $xy$, $uv$
of edges in $M$, the set $\{x,y,u,v\}$ does not induce a $C_4$. Then
$M$ is an optimal edge-identifying code and
$\EID(G)=\frac{|V(G)|}{2}$.
\end{corollary}

We note that in particular, if the girth of a graph $G$ is at least~5
and $G$ admits a perfect matching $M$, then $M$ is a minimum-size
identifying code of $G$. For example, the edge-identifying code of the
Petersen graph given in Figure~\ref{fig:petersen} is optimal.

As another application of Corollary~\ref{cor:PMgirth4}, we give the
edge-identifying code of all hypercubes of dimension $d\geq 4$.

\begin{proposition}\label{Formula:Hypercubes}
 For $d\geq 4$,  we have $\EID(\mathcal H_d)=2^{d-1}$.
\end{proposition}

\begin{proof}
By Theorem~\ref{prop:low2}, we have $\EID(\mathcal H_d)\geq 2^{d-1}$.
We will construct by induction a perfect matching $M_d$ of $\mathcal
H_d$ such that no pair of edges induces a $C_4$, for $d\geq 4$. By
Lemma~\ref{lem:girth4}, $M_d$ will be an edge-identifying code of
$\mathcal H_d$, proving the result. Two such matchings of $\mathcal
H_4$, which are also disjoint, are presented in
Figure~\ref{fig:match}.  The matching $M_5$ can now be built using
each of these two matchings of $\mathcal H_4$ --- one matching per copy
of $\mathcal H_4$ in $\mathcal H_5$. It is easily verified that $M_5$
has the required property. Furthermore, $M_5$ has the extra property
that for each edge $uv$ of $M_5$, $u$ and $v$ do not differ on the
first coordinate (recall that we build $\mathcal H_5$ from $\mathcal
H_4$ by adding a new coordinate on the left, hence the first
coordinate is a the new one).  We now build the matching $M_d$ of
$\mathcal H_d$ ($d\geq 6$) from $M_{d-1}$ in such a way that no two
edges of $M_d$ belong to a 4-cycle in $\mathcal H_d$ and that for each
edge $uv$ of $M_d$, $u$ and $v$ do not differ on the first
coordinate. To do this, let $\mathcal H'_1$ be the copy of $\mathcal
H_{d-1}$ in $\mathcal H_d$ induced by the set of vertices whose first
coordinate is~0. Similarly, let $\mathcal H'_2$ be the copy of
$\mathcal H_{d-1}$ in $\mathcal H_d$ induced by the other
vertices. Let $\mathcal M'_1$ be a copy of $M_{d-1}$ in $\mathcal
H'_1$ and let $\mathcal M'_2$ be a matching in $\mathcal H'_2$
obtained from $\mathcal M'_1$ by the following transformation: for
$e=uv\in\mathcal M'_1$, define $\psi(e)=\sigma(u)\sigma(v)$ where
$\sigma (x)= x+(1,0,0,\ldots,0)$.  It is now easy to check that the
new matching $M_d=\mathcal M'_1\cup\mathcal M'_2$ has both properties
we need.
\end{proof}

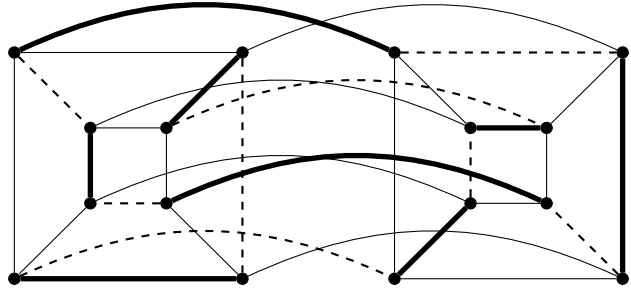
\begin{figure}[!ht]
\begin{center}
\begin{tikzpicture}[join=bevel,inner sep=0.5mm]
\node[graphnode](0000) at (0,0) {};
\node[graphnode](0001) at (0,1) {};
\node[graphnode](0010) at (1,0) {};
\node[graphnode](0011) at (1,1) {};

\node[graphnode](0100) at (-1,-1) {};
\node[graphnode](0101) at (-1,2) {};
\node[graphnode](0110) at (2,-1) {};
\node[graphnode](0111) at (2,2) {};

\draw[-] (0001)--(0011)--(0010);
\draw[-,dashed,thick] (0010)--(0000) (0111)--(0110) (0101)--(0001);
\draw[-] (0100)--(0101)--(0111);
\draw[-] (0100)--(0000);
\draw[-] (0111)--(0011);
\draw[-] (0110)--(0010);

\begin{scope}[xshift = +5cm]
\node[graphnode](1000) at (0,0) {};
\node[graphnode](1001) at (0,1) {};
\node[graphnode](1010) at (1,0) {};
\node[graphnode](1011) at (1,1) {};

\node[graphnode](1100) at (-1,-1) {};
\node[graphnode](1101) at (-1,2) {};
\node[graphnode](1110) at (2,-1) {};
\node[graphnode](1111) at (2,2) {};

\draw[-] (1001)--(1011)--(1010)--(1000);
\draw[-] (1100)--(1101)  (1111)--(1110)--(1100);
\draw[-,dashed, thick] (1101) -- (1111) (1000)--(1001) (1110)--(1010);
\draw[-] (1111)--(1011);
\draw[-] (1101)--(1001);
\end{scope}

\draw[line width=2pt, out=25,in=155] (0101) to (1101);
\draw[line width=2pt, out=25,in=155] (0010) to (1010);

\draw[edge, out=25,in=155,dashed,thick] (0100) to (1100);
\draw[edge, out=25,in=155,dashed,thick] (0011) to (1011);

\draw[edge, out=25,in=155] (0000) to (1000);
\draw[edge, out=25,in=155] (0001) to (1001);
\draw[edge, out=25,in=155] (0110) to (1110);
\draw[edge, out=25,in=155] (0111) to (1111);

\draw[line width=2pt] (0000)--(0001);
\draw[line width=2pt] (0011)--(0111);
\draw[line width=2pt] (0100)--(0110);
\draw[line width=2pt] (1110)--(1111);
\draw[line width=2pt] (1100)--(1000);
\draw[line width=2pt] (1011)--(1001);

\end{tikzpicture}
\caption{\label{fig:match} Two disjoint edge-identifying codes of $\mathcal
H_4$}
\end{center}
\end{figure}

We note that the formula of Proposition \ref{Formula:Hypercubes} does
not hold for $d=2$ and $d=3$. For $d=2$ the hypercube $\mathcal H_2$
is isomorphic to $C_4$ and thus $\EID(\mathcal H_2)=3$. For $d=3$, we
note that an identifying code of size~4, if it exists, must be a
matching with no pair of edges belonging to a 4-cycle.  But this is
not possible. An identifying code of size~5 is shown in
Figure~\ref{fig:H3}, therefore $\EID(\mathcal H_3)=5$.

\begin{figure}[!ht]
\begin{center}
\begin{tikzpicture}[join=bevel,inner sep=0.5mm]
\node[graphnode](000) at (0,0) {};
\node[graphnode](001) at (0,1) {};
\node[graphnode](010) at (1,0) {};
\node[graphnode](011) at (1,1) {};

\node[graphnode](100) at (-1,-1) {};
\node[graphnode](101) at (-1,2) {};
\node[graphnode](110) at (2,-1) {};
\node[graphnode](111) at (2,2) {};

\draw[edge]  (000)--(001)--(011)--(010)--(000);
\draw[edge]  (100)--(110)--(111)--(101)--(100);
\foreach \I in {00,01,10,11}
\draw[edge] (1\I)--(0\I);

\draw[line width=2pt] (100)--(101)--(111)--(011) (001)--(000) (010)--(011);

\end{tikzpicture}
\caption{\label{fig:H3} An optimal edge-identifying code of $\mathcal H_3$}
\end{center}
\end{figure}
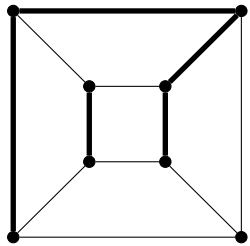

\section{Lower Bounds}\label{LowerBound}
Recall from Theorem~\ref{boundlogn} that $\ID(G)$ is bounded below by a function of
the order of $G$. As mentioned before, this bound is tight. Let $\mathcal C$ be a set of $c$
isolated vertices. We can build a graph $G$ of order $2^c-1$ such that $\mathcal C$ is an
identifying code of $G$.  To this end, for every
subset $X$ of $\mathcal C$ with $|X|\geq 2$, we associate a new vertex which is joined to
all vertices in $X$ and only to those vertices. Then, it is easily seen that $\mathcal C$ is an
identifying code of this graph. However, the graph built in this way is far from being a
line graph as it contains $K_{1,t}$, even for large values of $t$. In fact this lower bound
turns out to be far from being tight for the family of line graphs. 
In this section we
give a tight lower bound on the size of an edge-identifying code of a graph in terms of
the number of its edges. Equivalently we have a lower bound for the size of an identifying
code in a line graph in terms of its order. This lower bound is of the order $\Theta(\sqrt n)$
and thus is a much improved lower bound with respect to the general bound of
Theorem~\ref{boundlogn}.

Let $G$ be a pendant-free graph and let $\mathcal{C}_E$ be an edge-identifying code of $G$.
To avoid trivialities such as having isolated vertices we may assume $G$ is connected.
We note that this does not mean that
the subgraph induced by $\mathcal{C}_E$ is also connected, in fact we observe almost the contrary,
i.e. in most cases, an edge-identifying code of a minimum size  will induce a disconnected
subgraph of $G$. We first prove a lower bound for the case when an edge-identifying code
induces a connected subgraph.

\begin{theorem}\label{Thm:conecLow}
 If an edge-identifying code $\mathcal{C}_E$ of a nontrivial graph $G$ induces a connected
subgraph of $G$ which is not isomorphic to $K_2$, then $G$ has at most
$\binom{|\mathcal{C}_E|+2}{2}-4$ edges.
Furthermore, equality can only hold if $\mathcal{C}_E$ induces a path.
\end{theorem}

\begin{proof}
Let $G'$ be the subgraph induced by $\mathcal{C}_E$.
Since we assumed $G'$ is connected, and since $G'$ is pendant-free, it cannot have
three vertices. Since we assumed $G'\ncong K_2$, we conclude that $G'$ has at least
four vertices. For each vertex $x$ of $G'$, let $\mathcal{C}_E^x$ be the set of all edges incident
to $x$ in $G'$. Let $e=uv$ be an edge of $G$, then one or both of $u$ and $v$ must
be in $V(G')$. Therefore, depending on which of these vertices belong to $\mathcal{C}_E$, $e$ is uniquely determined by either $\mathcal{C}_E^u$ (if $u\in V(G')$ and $v\notin V(G')$), or $\mathcal{C}_E^v$ (if $u\notin V(G')$ and $v\in V(G')$),
or $\mathcal{C}_E^u\cup \mathcal{C}_E^v$ (if both $u,v\in V(G')$). The total number of sets of this form can be at most $|V(G')|+\binom{|V(G')|}{2}=\binom{|V(G')|+1}{2}$,
thus if $|V(G')|\leq |\mathcal{C}_E|$ we are done.
Otherwise, since $G'$ is connected, $|V(G')|=|\mathcal{C}_E|+1$ and $G'$ is a pendant-free tree on at
least 4 vertices. If $v$ is a vertex of degree~1 adjacent to $u$, then we have
$\mathcal{C}_E^v=\{uv\}$ but  $ uv \in \mathcal{C}_E^u$ and, therefore, $\mathcal{C}_E^u=\mathcal{C}_E^u\cup \mathcal{C}_E^v$.
On the other hand, by Lemma~\ref{degree2inPendantFree}, there are two vertices of degree~2 that have neighbours of degree~1. Let $u$ be such a vertex, let $v$ be its neighbour of
degree~1 and $x$ be its other neighbour. Then $\mathcal{C}_E^v=\{uv\}$ and $\mathcal{C}_E^u=\{uv, ux\}$ and,
therefore,  $\mathcal{C}_E^u\cup \mathcal{C}_E^x=\mathcal{C}_E^v\cup \mathcal{C}_E^x$. Thus the total number of distinct sets of the
form $\mathcal{C}_E^y$ or $\mathcal{C}_E^y\cup \mathcal{C}_E^z$ is at most $\binom{|\mathcal{C}_E|+2}{2}-4$. But if equality holds
there can only be two vertices of degree~1  in $G'$ and hence $\mathcal{C}_E$ is a path.
\end{proof}

We note that if this bound is tight, then $G'$ is a path. Furthermore, for each path
$P_{k+1}$ one can build many graphs which have $P_{k+1}$ as an edge-identifying code
and have $\binom{k+2}{2}-4$ edges. The set of all these graphs will be denoted by
$\mathcal J_k$. An example of such a graph is obtained from $K_{k+2}$ by removing
a certain set of four edges as shown in Figure~\ref{fig:Ji}. Note that every other member
of $\mathcal J_k$ is obtained from the previous example by splitting the vertex that
does not belong to $P_{k+1}$ (but without adding any new edge).

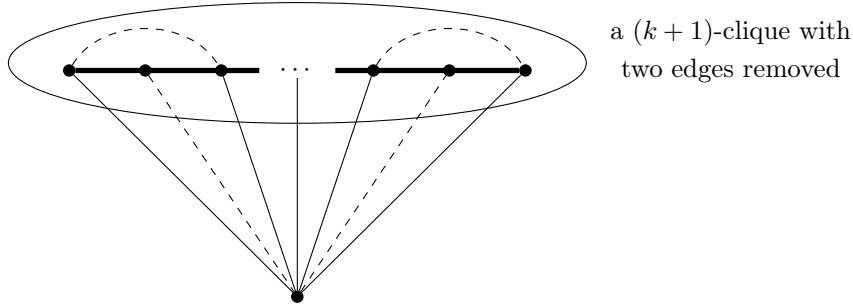
\begin{figure}[!ht]
\begin{center}
\begin{tikzpicture}[join=bevel,inner sep=0.5mm]
\draw[edge, draw = black] (4,3.1) ellipse(3.8 and 0.8) node{};
\node at (9.7,3.5) {a $(k+1)$-clique with};
\node at (9.7,3) {two edges removed};
\foreach \i in {1,2,3,5,6,7}
  \node[draw,shape=circle,fill=black,draw=black,minimum size=0.6pt,inner sep=1.5pt] (x\i) at (\i,3) {};
\draw[-, line width=2pt] (x1) -- (3.5,3) (4.5,3) -- (x7);
\node at (4,3) {$\cdots$};
\node[draw,shape=circle,fill=black,draw=black,minimum size=0.6pt,inner sep=1.5pt] (v) at (4,0) {};
\foreach \i in {1,3,5,7}
\draw[edge] (v) -- (x\i);
\draw[edge] (v) -- (4,2.9);
\draw[-,dashed] (x2) -- (v) -- (x6);
\draw[-, dashed, out=60,in=120] (x1) to (x3) (x5) to (x7);
%\draw[-] (x1) -- (2,3.8) (x2) -- (3,3.8) (x3) -- (4,3.8);
\end{tikzpicture}
\caption{\label{fig:Ji} An extremal graph of $\mathcal{J}_k$ with its connected edge-identifying code}
\end{center}
\end{figure}

Next we consider the case when the subgraph induced by $\mathcal{C}_E$ is not necessarily connected. 

\begin{theorem}\label{Thm:Low1}
 Let $G$ be a pendant-free graph and let $\mathcal{C}_E$ be an edge-identifying code of
$G$ with $|\mathcal{C}_E|=k$. Then we have:

 \begin{equation*} |E(G)| \leq 
  \begin{cases} 
  {\binom{\frac{4}{3} k }{2}},  & \mbox{if } k\equiv 0 \bmod 3 \\
  {\binom{\frac{4}{3} (k-1)+1}{2}}+1,  & \mbox{if } k\equiv 1 \bmod 3 \\
  {\binom{\frac{4}{3} (k-2)+2}{2}}+2,  & \mbox{if } k\equiv 2 \bmod 3. \\
  \end{cases}
 \end{equation*}

\end{theorem}

\begin{proof}
Let $G$ be a graph with maximum number of edges among all graphs with $\EID(G)=k$. It can be easily checked that for $k=1,2$ or $3$, the maximum number of edges of $G$ is~1, 3 or~6 respectively. For $k\geq 4$, we prove a slightly stronger statement: given an edge-identifying
code $\mathcal{C}_E$ of $G$ of size $k$, all but at most two of the connected components of the subgraph induced
by $\mathcal{C}_E$  must be isomorphic to $P_4$.
When there is only one component not isomorphic to $P_4$, it must be isomorphic to a
$P_2$, a $P_5$ or a $P_6$. If there are two such components, then they can be two copies
of $P_2$, a $P_2$ with a $P_5$, or just two copies of $P_5$. This depends on the
value of $k \bmod 3$.

To prove our claim let $G$ be a graph as defined above, let $\mathcal{C}_E$ be an edge-identifying
code of size~$k$ of $G$ and let $G'$ be the subgraph induced by $\mathcal{C}_E$.
For each vertex $u\in V(G)\setminus V(G')$, we can assume that $u$ has degree~1:
if $u$ has degree $d>1$,  with neighbours $v_1, \ldots, v_d$ necessarily in $V(G')$,
then replace $u$ by $d$ vertices of degree~1:
$u_1, \ldots, u_d$, connecting $u_i$ to $v_i$.
Then the number of edges does not change, and the code $\mathcal{C}_E$ remains an
edge-identifying code of size~$k$, thus it suffices to prove our claim for this new graph.
Let $G'_1, G'_2, \ldots, G'_r$ be the connected components of $G'$ with $|V(G'_i)|=n'_i$.
For each $i\in \{1,\ldots,r\}$, let $G_i$ be the graph induced by the vertices of $G'_i$ 
and the vertices connected to $G'_i$ only.
To each vertex $x$ of $G'$ we assign the set $\mathcal{C}_E^x$ of edges in $G'$ incident to $x$.

We first note that no $G'_i$ can be of order~3, because there is no connected
pendant-free graph on three vertices. If $u$ and $v$ are vertices from two disjoint
components of $G'$ with each component being of order at least~4, then the pair $u,v$
is uniquely determined by $\mathcal{C}_E^u\cup \mathcal{C}_E^v$, thus by maximality of $G$, $uv$ is an edge of $G$.
If a component of $G'$ is isomorphic to $K_2$, assuming $u$ and
$u'$ are vertices of this component, then for any other vertex $v$ of $G'$ exactly one of $uv$
or $u'v$ is an edge of $G$.

We now claim that each $G'_i$ with $n'_i\geq 4$ is a path. By
contradiction, if a $G'_i$ is not a path, we replace $G_i$ by a
member $J_{n'_i-1}$ of $\mathcal J_{n'_i-1}$ with $P_{n'_i}$ being its
edge-identifying code. Then we join each vertex of $P_{n'_i}$ to each
vertex of each $G'_j$ (with $j\neq i$ and $n'_j \geq 4$) and to exactly
one vertex of each $G_j$ with $n'_j=2$.  We note that the new graph 
still admits an edge-identifying code of size~$k$. However, it has more edges than $G$.
Indeed, while the number of edges connecting $G'_i$ and the $G'_j$'s
($j\neq i$) is not decreased, the number of edges in $G_i$ is
increased when we replace $G_i$ by $J_{n'_i-1}$.  This can be seen by
applying Theorem~\ref{Thm:conecLow} on $G_i$.

We now show that none of the $G'_i$'s can  have more than six vertices.
By contradiction, suppose $G'_1$ is a component with  $n'_1\geq 7$ vertices 
(thus $n'_1-1$ edges). We build a new graph $G_1^*$ from $G$ as follows. 
We take disjoint copies of $J_3 \in \mathcal J_3$ and
$J_{n'_1-4} \in \mathcal J_{n'_1-4}$ with $P_4$ and $P_{n'_1-3}$ being, respectively,
their edge-identifying codes.
We now let $V(G_1^*)= V(J_3)\cup V(J_{n'_1-4}) \cup (V(G)\setminus V(G_1))$.
The edges of $J_3$, $J_{n'_1-4}$ and $G-G_1$ are also edges of $G_1^*$.
We then add edges between these three parts as follows. We join every vertex of
$P_4$ to each vertex of $P_{n'_1-3}$. For $i=2, 3, \ldots, r$ if
$n'_i \geq 4$, join every vertex of $G'_i$ to each vertex of $P_4 \cup P_{n'_1-3}$.
If $n'_i=2$, we choose exactly one vertex of $G'_i$ and join it to each vertex of
$P_4 \cup P_{n'_1-3}$. The construction of $G^*_1$ ensures that it still admits an
edge-identifying code of size~$k$, but it has more edges than $G$. In fact, the number
of edges is increased in two ways. First, because $P_4 \cup P_{n'_1-3}$ has one more
vertex than $G'_1$, the number of edges connecting $P_4 \cup P_{n'_1-3}$ to $G-G_1$
has increased (unless $r=1$). More importantly, the number of edges induced by
$J_3 \cup J_{n'_1-4}$ is
$6+\binom{n'_1-2}{2}-4+ 4\times (n'_1-3)=\frac{{n'_1}^2}{2}+\frac{3n'_1}{2} -7$ which
is strictly more than $|E(G'_1)|=\frac{{n'_1}^2}{2}+\frac{n'_1}{2} -4$ for $n'_1\geq 3$.
Since $n'_1\geq 7$, this contradicts the maximality of $G$.

With a similar method, the following transformations strictly increase the number of edges while the new graph still admits an edge-identifying code of size~$k$:
\begin{enumerate}

\item Two components of $G'$ each on six vertices transform
into two graphs of $\mathcal{J}_3$ and a graph of $\mathcal J_4$.

\item One component  of $G'$ on six vertices and another component on five vertices
transform  into three graphs of $\mathcal J_3$.

\item One component of $G'$ on six vertices and one on two vertices transform
 into two graphs of $\mathcal J_3$.

\item Three components of $G'$ each on five vertices transform into four graphs of $\mathcal J_3$.

\item Two components of $G'$ on five vertices and one on two vertices transform into three 
graphs of $\mathcal J_3$.

\item A component of $G'$ on five vertices and two on two vertices transform into two graphs of $\mathcal J_4$.

\item Three components of $G'$ each isomorphic to $P_2$ transform into a graph
of $\mathcal J_3$.
\end{enumerate}

%We finally  notice that each of the following pairs would contribute the same number of edges
%to form $G$:  $\{P_4\cup P_2, P_5\}$,  $\{P_4\cup P_2\cup P_2, P_6\}$.

For the proof of case 7, we observe that the number of edges
identified by the three $P_2$'s would be the same as the number of
edges identified by the $P_4$. However, since $k\geq 4$, there must be
some other component in $G'$. Moreover, the number of vertices of the
three $P_2$'s, which are joined to the vertices of the other
components of $G'$, is three, whereas the number of these vertices of
the $P_4$, is four. Hence the maximality of $G$ is contradicted.

We note that cases~1, 2 and~3 imply that if a component of $G'$ is
isomorphic to $P_6$, every other component is isomorphic to
$P_4$. Then cases~4, 5 and~6 imply that if a component is isomorphic
to $P_5$, then at most one other component is not isomorphic to $P_4$
and such component is necessarily either a $P_2$ or a $P_5$. Finally,
case~7 shows that there can be at most two components both isomorphic
to $P_2$.

We conclude that each of the components of $G'$ is isomorphic to $P_4$ except for possibly two of them.
These exceptions are dependent on the value of $k \bmod 3$ as we described. 
The formulas of the theorem can be derived using these structural properties of $G$.
For instance, in the case $k \equiv 0 \bmod 3$, each component of $G'$ is
isomorphic to $P_4$. There are $\tfrac{k}{3}$ such components. For each
component $G'_i$, there are six edges in the graph $G_i$. That gives $2k$ edges.
The other edges of $G$ are edges between two components of $G'$. By maximality
of $G$, between two components of $G'$, there are exactly 16 edges. There are
$\displaystyle\binom{\frac{k}{3}}{2}$ pairs of components of $G'$. Hence, the
number of edges in $G$ is:

$$2k+16\displaystyle\binom{\frac{k}{3}}{2}=\displaystyle\binom{\frac{4}{3}k}{2}
.$$

The other cases can be proved with the same method. 
\end{proof}

We note that this bound is tight and the examples were in fact built inside the proof.
More precisely, for $k\equiv 0 \bmod 3$ we take $\frac {k}{3}$ disjoint copies of elements
of $\mathcal J_3$ each having a $P_4$ as an edge-identifying code. We then add an edge
between each pair of vertices coming from two distinct such $P_4$'s. We note that the union
of these
$P_4$'s is a minimum edge-identifying code of the graph.
If $k\not\equiv 0 \bmod 3$, then we build a similar construction. This time we use elements
from $\mathcal J_3$ with at most two exceptions that are elements of $\mathcal J_4$ or
$\mathcal J_5$.

The above theorem can be restated in the language of line graphs as follows.

\begin{corollary}\label{cor:LowerBoundSqrt}
Let $G$ be a twin-free line graph on $n\geq 4$ vertices. Then we have 
$\gamma^{\text{\tiny {ID}}}(G)\geq \frac{3\sqrt{2}}{4}\sqrt{n}$.
\end{corollary}

\begin{proof}
Suppose $G$ is the line graph of a pendant-free graph $H$
($\mathcal{L}(H)=G$). Let $k=\gamma^{\text{\tiny
    {ID}}}(G)=\gamma^{\text{\tiny {EID}}}(H)$, and let $n$ be the
number of vertices of $G$ ($n=|E(H)|$). Then, after solving the quadratic
inequalities of Theorem~\ref{Thm:Low1} for $k$, we have:
\begin{align*}
k\geq & \frac{3}{8}+\frac{3\sqrt{8n+1}}{8} \text{, for } k\equiv 0\bmod 3,\\
k\geq &\frac{5}{8}+\frac{3\sqrt{8n-7}}{8} \text{, for } k\equiv 1\bmod 3,\\
k\geq &\frac{3}{8}+\frac{3\sqrt{8n-15}}{8} \text{, for } k\equiv 2\bmod 3.\\
\end{align*}
It is then easy to check that the right-hand side of each of the three
inequalities is at least as $\frac{3\sqrt{2}}{4}\sqrt{n}$ for $n\geq 3$.
\end{proof}

{\bf Remark.} Note that the lower bound of $\ID(G) \geq \Theta
(\sqrt{|V(G)|})$, which holds for the class of line graphs, is also
implied by Theorem~\ref{prop:low2}. However, the bound of \ref{cor:LowerBoundSqrt} is more precise. 
In~\cite{B70}, Beineke characterized line
graphs by a list of nine forbidden induced subgraphs. Considering Beineke's
characterization, the lower bound of Corollary~\ref{cor:LowerBoundSqrt} can be
restated as follows: $\ID(G) \geq \Theta (\sqrt{|V(G)|})$ holds if $G$
has no induced subgraph from Beineke's list. It is then natural to ask what is a minimal list of forbidden induced subgraphs for which a similar claim would hold. Note that the
claw graph, $K_{1,3}$, belongs to Beineke's list of forbidden
subgraphs. However, we remark that the bound $\ID(G) \geq \Theta (\sqrt{|V(G)|})$ does not hold for the
class of claw-free graphs. Examples can be built as follows: let $A$
be a set of size~$k$ and let $B$ be the set of nonempty subsets of
$A$. Let $G$ be the graph built on $A\cup B$, where $A$ and $B$ each
induce a complete graph and a vertex $a$ of $A$ is joined to a vertex
$b$ of $B$ if $a\in b$. This graph is claw-free and it is easy to find
an identifying code of size at most $2k=\Theta (\log{|V(G)|})$ in $G$.

\section{Upper bounds}\label{UpperBound}

The most natural question in the study of identifying codes in graphs
is to find an identifying code as small as possible. A general bound,
only in terms of the number of vertices of a graph, is provided by
Theorem \ref{boundn-1}. Furthermore, the class of all graphs with
$\ID(G)=|V(G)|-1$ is classified in \cite{FGKNPV10}.  It is easy to
check that none but six of these graphs are line graphs. Thus we have
the following corollary (where $G\join H$ denotes the complete join of
graphs $G$ and $H$):

\begin{corollary}\label{cor:extremalgraphs}
If $G$ is a twin-free line graph with 
$G\notin \{P_3,P_4, C_4, P_4\join K_1, C_4\join K_1,\mathcal L(K_4)\}$, then we have
$\ID(G)\leq |V(G)|-2$.
\end{corollary}

Since $\EID(K_{2,n})=2n-2$, $\ID(\mathcal L
(K_{2,n}))=|V(\mathcal{L}(K_{2,n}))|-2$ and the bound of
Corollary~\ref{cor:extremalgraphs} is tight for an infinite family of
graphs. Conjecture~\ref{ConjFKKR} proposes a better bound in terms of
both the number of vertices and the maximum degree of a graph.  As
pointed out in Proposition~\ref{LineGraphofRegulars}, most of the known
extremal graphs for Conjecture~\ref{ConjFKKR} are line graphs. In this
section, after proving some general bound for the edge-identifying
code number of a pendant-free graph we will show that
Conjecture~\ref{ConjFKKR} holds for the class of line graphs of high
enough density.

We recall that a graph on $n$ vertices is \emph{2-degenerated} if its vertices can
be ordered $v_1, v_2, \ldots, v_n$ such that each vertex $v_i$ is joined to
at most two vertices in $\{v_1, v_2, \ldots, v_{i-1}\}$. Our main idea for
proving upper bounds is to show that
given a pendant-free graph $G$, any (inclusionwise) minimal edge-identifying
code $\mathcal{C}_E$ induces a 2-degenerated subgraph of $G$ and hence
$|\mathcal{C}_E|\leq 2|V(G)|-3$. Our proofs are constructive and one could build
such small edge-identifying codes.

\begin{theorem}\label{EIDare2Degenerated}
 Let $G$ be a pendant-free graph and let $\mathcal{C}_E$ be a minimal
 edge-identifying code of $G$.  Then $G'$, the subgraph induced by
 $\mathcal{C}_E$, is 2-degenerated.
\end{theorem}

\begin{proof}
 Let $uv$ be an edge of $G'$ with $d_{G'}(u), d_{G'}(v)\geq 3$. By minimality of $\mathcal{C}_E$
the subset $\mathcal C'=\mathcal{C}_E-uv$ of $E(G)$ is not an edge-identifying code of $G$. By the choice of 
$uv$, $\mathcal C'$ is still an edge-dominating set, thus there must be two edges, $e_1$ and $e_2$,
that are not separated by $\mathcal C'$. Hence one of them, say $e_1$, is incident either to $u$
or to $v$ (possibly to both) and the other one ($e_2$) is incident to neither one.

We consider two cases: either $e_1=uv$ or $e_1$ is incident  to only one
of $u$ and $v$. In the first case, $e_2$ is adjacent to every edge of $\mathcal{C}'$ which $uv$ is
adjacent to. Since for each vertex of $uv$ there are at least two edges in $\mathcal{C}'$ incident to
this vertex, the subgraph induced by $u$, $v$ and the vertices of $e_2$ must be isomorphic
to $K_4$ and there should be no other edge of $\mathcal{C}'$ incident to any
vertex of this $K_4$ (see Figure~\ref{fig:2degeneratedC1}).

In the other case, suppose $e_1$ is adjacent to $uv$ at $u$. Let $x$
and $y$ be two neighbours of $u$ in $G'$ other than $v$. Then it
follows that $e_2=xy$ and, therefore, $d_{G'}(u)=3$. Let $z$ be the
other end of $e_1$. We consider two subcases: either $z\notin
\{x,y\}$, or, without loss of generality, $z=x$. Suppose $z\notin
\{x,y\}$.  Recall that $uv$ is the only edge separating $e_1$ and
$e_2$, but $e_1$ must be separated from $ux$. Thus $zy\in
\mathcal{C}_E$. Similarly, $e_1$ must be separated from
$uy$, so $zx\in\mathcal{C}_E$. Furthermore,
$d_{G'}(x)=d_{G'}(y)=d_{G'}(z)=2$ and $\{x,y,z, u\}$ induces a $C_4$
in $G'$ (see Figure~\ref{fig:2degeneratedC2a}).  Now suppose $e_1=ux$,
since $uv$ is the only edge separating $e_1$ and $e_2$, then $uy$ and
possibly $xy$ are the only edges in $G'$ incident to $y$, so
$d_{G'}(y)\leq 2$ and $d_{G'}(u)=3$ (see
Figures~\ref{fig:2degeneratedC2b} and~\ref{fig:2degeneratedC2c}).

\begin{figure}[!ht]
\begin{center}
\subfigure[]{
\begin{tikzpicture}[join=bevel,inner sep=0.5mm,scale=0.7]
\node[draw,shape=circle,fill=black,draw=black,minimum size=0.6pt,inner sep=1.5pt] (a) at (1.5,3){};
\node[draw,shape=circle,fill=black,draw=black,minimum size=0.6pt,inner sep=1.5pt] (b) at (1.5,1){};
\node[draw,shape=circle,fill=black,draw=black,minimum size=0.6pt,inner sep=1.5pt] (u) at (0,0){};
\node at (-0.35,0) {$u$};
\node[draw,shape=circle,fill=black,draw=black,minimum size=0.6pt,inner sep=1.5pt] (v) at (3,0){};
\node at (3.3,0) {$v$};
\draw[-,line width=2pt] (a)--(u)--(b)--(v)--(a) (u) -- (v);
\draw[-] (b) -- (a);
\node[shape=circle,fill=white] at (1.5,1.8) {$e_2$};
\node[shape=circle,fill=white] at (1.5,0) {$e_1$};
\end{tikzpicture}
\label{fig:2degeneratedC1}
}
\subfigure[]{
\begin{tikzpicture}[join=bevel,inner sep=0.5mm,scale=0.7]
\node[draw,shape=circle,fill=black,draw=black,minimum size=0.6pt,inner sep=1.5pt] (u) at (1.5,3){};
\node at (1.5,3.3) {$u$};
\node[draw,shape=circle,fill=black,draw=black,minimum size=0.6pt,inner sep=1.5pt] (x) at (1.5,1){};
\node at (1.5,0.7) {$x$};
\node[draw,shape=circle,fill=black,draw=black,minimum size=0.6pt,inner sep=1.5pt] (y) at (0,0){};
\node at (-0.35,0) {$y$};
\node[draw,shape=circle,fill=black,draw=black,minimum size=0.6pt,inner sep=1.5pt] (z) at (3,0){};
\node at (3.3,0) {$z$};
\node[draw,shape=circle,draw=black,minimum size=0.6pt,inner sep=1.5pt] (v) at (3,3){};
\node at (3,3.3) {$v$};

\draw[-,line width=2pt] (4.2,2.5) -- (v) -- (4.2,3.5);
\node[rotate=90,scale=0.7] at (4,3) {$\cdots$};

\draw[-,line width=2pt] (u)--(y)--(z)--(x)--(u)--(v);
\draw[-] (y)--(x) (u)--(z);
\node[shape=circle,fill=white] at (0.8,0.5) {$e_2$};
\node[shape=circle,minimum size=15pt] at (1.5,0) {};
\node[shape=circle,fill=white] at (2.25,1.5) {$e_1$};
\end{tikzpicture}
\label{fig:2degeneratedC2a}
}
\subfigure[]{
\begin{tikzpicture}[join=bevel,inner sep=0.5mm,scale=0.7]
\node[draw,shape=circle,fill=black,draw=black,minimum size=0.6pt,inner sep=1.5pt] (u) at (1.5,3){};
\node at (1.5,3.3) {$u$};
\node[draw,shape=circle,fill=black,draw=black,minimum size=0.6pt,inner sep=1.5pt] (y) at (0,0){};
\node at (-0.35,0) {$y$};
\node[draw,shape=circle,draw=black,minimum size=0.6pt,inner sep=1.5pt] (z) at (3,0){};
\node at (3.8,0) {$z=x$};

\node[draw,shape=circle,draw=black,minimum size=0.6pt,inner sep=1.5pt] (v) at (3,3){};
\node at (3,3.3) {$v$};

\draw[-,line width=2pt] (4.2,2.5) -- (v) -- (4.2,3.5);
\node[rotate=90,scale=0.7] at (4,3) {$\cdots$};

\draw[-,line width=2pt] (y)--(u)--(z) (u)--(v);
\draw[-] (y)--(z);
\node[shape=circle,fill=white] at (1.5,0) {$e_2$};
\node[shape=circle,fill=white] at (2.25,1.45) {$e_1$};
\end{tikzpicture}
\label{fig:2degeneratedC2b}
}
\subfigure[]{
\begin{tikzpicture}[join=bevel,inner sep=0.5mm,scale=0.7]
\node[draw,shape=circle,fill=black,draw=black,minimum size=0.6pt,inner sep=1.5pt] (u) at (1.5,3){};
\node at (1.5,3.3) {$u$};
\node[draw,shape=circle,fill=black,draw=black,minimum size=0.6pt,inner sep=1.5pt] (y) at (0,0){};
\node at (-0.35,0) {$y$};
\node[draw,shape=circle,draw=black,minimum size=0.6pt,inner sep=1.5pt] (z) at (3,0){};
\node at (3.8,0) {$z=x$};

\node[draw,shape=circle,draw=black,minimum size=0.6pt,inner sep=1.5pt] (v) at (3,3){};
\node at (3,3.3) {$v$};

\draw[-,line width=2pt] (4.2,2.5) -- (v) -- (4.2,3.5);
\node[rotate=90,scale=0.7] at (4,3) {$\cdots$};

\draw[-,line width=2pt] (z)--(y)--(u)--(z) (u)--(v);
\node[shape=circle,fill=white] at (1.5,0) {$e_2$};
\node[shape=circle,fill=white] at (2.25,1.45) {$e_1$};
\end{tikzpicture}
\label{fig:2degeneratedC2c}
}
\caption{\label{fig:2degenerated} Case distinctions in the proof of
Theorem~\ref{EIDare2Degenerated}. Black vertices have fixed degree in $G'$.
Thick edges belong to $\mathcal{C}_E$.}
\end{center}
\end{figure}
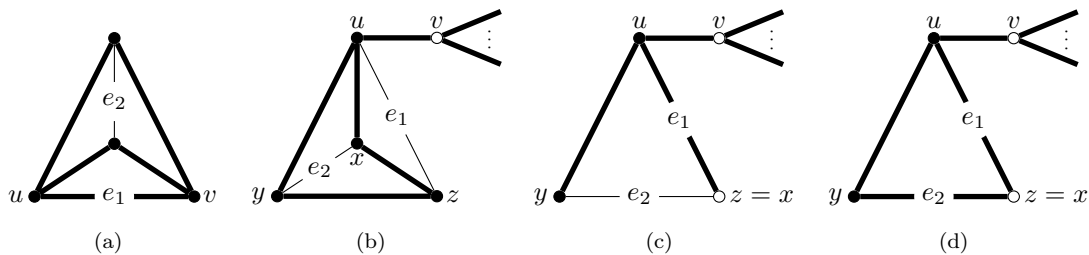

To summarize, we proved that given an edge $uv$, in a minimal edge-identifying code $\mathcal{C}_E$,
we have one of the following cases.
\begin{itemize}
\item One of $u$ or $v$ is of degree at most~2 in $G'$.
\item Edge $uv$ is an edge of a connected component of $G'$ isomorphic to $K_4^-$
(that is $K_4$ with an edge removed), see Figure~\ref{fig:2degeneratedC1}.
\item $d_{G'}(u)=3$ (considering the symmetry between $u$ and $v$) in which case either $u$ is incident to a $C_4$ whose other vertices are of degree~2 in $G'$ (Figure~\ref{fig:2degeneratedC2a}), or to a vertex of degree
1 in $G'$ (Figure~\ref{fig:2degeneratedC2b}) or to a triangle with one vertex $y$ of degree~2 in $G'$ and $y$ is not adjacent to $v$ (Figure~\ref{fig:2degeneratedC2c}). %A
\end{itemize}

In either case, there exists a vertex $x$ of degree at most~2 in $G'$
such that when $x$ is removed, at least one of the vertices $u$, $v$ has degree at
most~2 in the remaining subgraph of $G'$. In this way we can define an order
of elimination of the vertices of $G'$ showing that $G'$ is 2-degenerated.
\end{proof}

By further analysis of our proof we prove the following:

\begin{corollary}\label{Bound2|V(G)|-5}
 If $G$ is a pendant-free graph on $n$ vertices not isomorphic to $K_4$ or $K_4^-$,
 then $\EID (G)\leq 2n-5$.  
\end{corollary}
\begin{proof}
We first prove that if $G$ is a pendant-free graph on $n$ vertices not
isomorphic to $K_4$, then $\EID (G)\leq 2n-4$. Let $\mathcal{C}_E$ be
a minimal edge-identifying code and let $G'$ be the subgraph induced
by $\mathcal{C}_E$. Then, by Theorem~\ref{EIDare2Degenerated}, $G'$ is
2-degenerated. Let $v_n, v_{n-1},\ldots, v_1$ be a sequence of
vertices of $G'$ obtained by a process of eliminating vertices of
degree at most~2. Since $v_1$ and $v_2$ can induce at most a $K_2$, we
notice that there could only be at most $2n-3$ edges in
$G'$. Furthermore, if there are exactly $2n-3$ edges in $G'$, then
$v_1v_2\in\mathcal{C}_E$ and each vertex $v_i$, $3\leq i\leq n$, has exactly
two neighbours in $\{v_1,\ldots,v_{i-1}\}$. Hence, the subgraph
induced by $\{v_1,v_2,v_3,v_4\}$ is isomorphic to $K_4^-$. Considering
symmetries, there are three possibilities for the subgraph induced by
$\{v_1,\ldots, v_5\}$ (recall that $v_5$ is of degree~2 in this subgraph): see
Figure~\ref{fig:degen3} . In each of these
three cases, the edge $uv$ has both ends of degree at least~3. Thus,
we can apply the argument used in the proof of
Theorem~\ref{EIDare2Degenerated} on $G'$ and $uv$, showing that we
have one of the four configurations of
Figure~\ref{fig:2degenerated}. But none of them matches with the
configurations of Figure~\ref{fig:degen3}, a contradiction.

\begin{figure}[!ht]
 \begin{center}
\begin{tikzpicture}[join=bevel,inner sep=0.5mm]
 \node[draw,shape=circle,draw=black,minimum size=0.6pt,inner sep=1.5pt](1) at
(-0.5,0) {};
  \node[draw,shape=circle,draw=black,minimum size=0.6pt,inner sep=1.5pt](2) at
(1.5,0) {};
   \node[draw,shape=circle,draw=black,minimum size=0.6pt,inner sep=1.5pt](3) at
(0.5,-0.4) {};
    \node[draw,shape=circle,draw=black,minimum size=0.6pt,inner sep=1.5pt](4) at
(0.5,0.4) {};
      \node[] at (-0.8,0) {$u$};
      \node[] at (1.8,0) {$v$};
     \node[draw,shape=circle,draw=black,minimum size=0.6pt,inner sep=1.5pt](5)
at (0.5,1.3) {};
 \draw[line width=2pt] (4)--(1)--(3)--(2)--(4) (1)--(2);
\draw[line width=2pt] (1)--(5)--(2);
 \end{tikzpicture}
\hfil
 \begin{tikzpicture}
 \node[draw,shape=circle,draw=black,minimum size=0.6pt,inner sep=1.5pt](1) at
(0,0) {};
  \node[draw,shape=circle,draw=black,minimum size=0.6pt,inner sep=1.5pt](2) at
(1,0) {};
   \node[draw,shape=circle,draw=black,minimum size=0.6pt,inner sep=1.5pt](3) at
(1,1) {};
    \node[draw,shape=circle,draw=black,minimum size=0.6pt,inner sep=1.5pt](4) at
(0,1) {};
     \node[draw,shape=circle,draw=black,minimum size=0.6pt,inner sep=1.5pt](5)
at (0.5,1.8) {};
 \node[] at (-0.3,1) {$u$};
      \node[] at (1.3,1) {$v$};
 \draw[line width=2pt] (4)--(1)--(2)--(3)--(5)--(4)--(3)--(1);
 \end{tikzpicture}
\hfil
  \begin{tikzpicture}
 \node[draw,shape=circle,draw=black,minimum size=0.6pt,inner sep=1.5pt](1) at
(-0.5,0) {};
  \node[draw,shape=circle,draw=black,minimum size=0.6pt,inner sep=1.5pt](2) at
(1.5,0) {};
   \node[draw,shape=circle,draw=black,minimum size=0.6pt,inner sep=1.5pt](3) at
(0.5,-0.4) {};
    \node[draw,shape=circle,draw=black,minimum size=0.6pt,inner sep=1.5pt](4) at
(0.5,0.4) {};
    \node[] at (0.5,0.7) {$u$};
      \node[inner sep=0] at (0.5,-0.7) {$v$};
     \node[draw,shape=circle,draw=black,minimum size=0.6pt,inner sep=1.5pt](5)
at (0.5,1.3) {};
 \draw[line width=2pt] (4)--(1)--(3)--(2)--(4)--(3);
\draw[line width=2pt] (1)--(5)--(2);
 \end{tikzpicture}
 \caption{\label{fig:degen3} The three maximal $2$-degenerated graphs on five
vertices}
 \end{center}
 \end{figure}
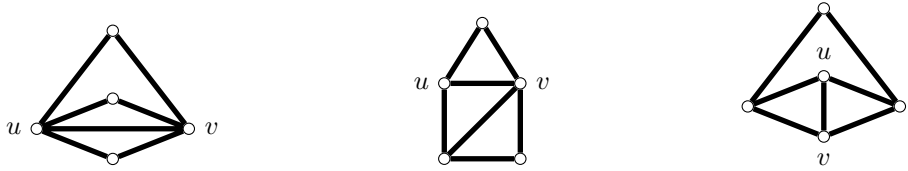

Now we show that if $\EID(G)=2n-4$, then $G\cong K^-_4$. This can be
easily checked if $G$ has at most four vertices, so we may assume
$n\geq 5$.  Let $G''$ be the subgraph of $G'$ induced by
$\{v_1,v_2,v_3,v_4, v_5\}$.  If $G''$ has seven edges, then it is isomorphic to
one of the graphs of Figure~\ref{fig:degen3}, and we are done just like in the
last case. Therefore, we can assume that $G''$ has exactly six edges
and, since it is 2-degenerated, by an easy case analysis, it must be
isomorphic to one of the graphs of Figure~\ref{fig:graphe5}.

 \begin{figure}[!ht]
 \begin{center}
 \begin{tikzpicture}
 \node[draw,shape=circle,draw=black,minimum size=0.6pt,inner sep=1.5pt](1) at
(0,0) {};
  \node[draw,shape=circle,draw=black,minimum size=0.6pt,inner sep=1.5pt](2) at
(1,0) {};
   \node[draw,shape=circle,draw=black,minimum size=0.6pt,inner sep=1.5pt](3) at
(1,1) {};
    \node[draw,shape=circle,draw=black,minimum size=0.6pt,inner sep=1.5pt](4) at
(0,1) {};
     \node[draw,shape=circle,draw=black,minimum size=0.6pt,inner sep=1.5pt](5)
at (0.5,1.8) {};
 \node[] at (-0.3,1) {$u$};
      \node[] at (1.3,1) {$v$};
 \draw[-,line width=2pt] (4)--(1)--(2)--(3)--(5)--(4)--(3);
 
 \node at (0.5,-0.5) {\footnotesize (i)};
 \node at (2,0) {};
 \end{tikzpicture}
  \begin{tikzpicture}
 \node[draw,shape=circle,draw=black,minimum size=0.6pt,inner sep=1.5pt](1) at
(0,0) {};
  \node[draw,shape=circle,draw=black,minimum size=0.6pt,inner sep=1.5pt](2) at
(1,0) {};
   \node[draw,shape=circle,draw=black,minimum size=0.6pt,inner sep=1.5pt](3) at
(1,1) {};
    \node[draw,shape=circle,draw=black,minimum size=0.6pt,inner sep=1.5pt](4) at
(0,1) {};
      \node[] at (-0.3,0) {$u$};
      \node[] at (1.3,1) {$v$};
     \node[draw,shape=circle,draw=black,minimum size=0.6pt,inner sep=1.5pt](5)
at (0.5,1.8) {};
 \draw[-,line width=2pt] (4)--(1)--(2)--(3)--(4);
 \draw[-,line width=2pt] (3)--(1);
  \draw[-,line width=2pt] (5)--(4);
 
 \node at (0.5,-0.5) {\footnotesize (ii)};
 \node at (2,0) {};
 \end{tikzpicture}
  \begin{tikzpicture}
  \node[draw,shape=circle,draw=black,minimum size=0.6pt,inner sep=1.5pt](1) at
(0,0) {};
  \node[draw,shape=circle,draw=black,minimum size=0.6pt,inner sep=1.5pt](2) at
(1,0) {};
   \node[draw,shape=circle,draw=black,minimum size=0.6pt,inner sep=1.5pt](3) at
(1,1) {};
    \node[draw,shape=circle,draw=black,minimum size=0.6pt,inner sep=1.5pt](4) at
(0,1) {};
     \node[draw,shape=circle,draw=black,minimum size=0.6pt,inner sep=1.5pt](5)
at (0.5,1.8) {};
  
  \node[] at (-0.3,0) {$u$};
      \node[] at (1.3,1) {$v$};
      
  \draw[-,line width=2pt] (4)--(1)--(2)--(3)--(4);
 \draw[-,line width=2pt] (3)--(1);
  \draw[-,line width=2pt] (5)--(3);

 \node at (0.5,-0.5) {\footnotesize (iii)};
 \node at (2,0) {};
 \end{tikzpicture}
  \begin{tikzpicture}
 \node[draw,shape=circle,draw=black,minimum size=0.6pt,inner sep=1.5pt](1) at
(0,0) {};
  \node[draw,shape=circle,draw=black,minimum size=0.6pt,inner sep=1.5pt](2) at
(1,0) {};
   \node[draw,shape=circle,draw=black,minimum size=0.6pt,inner sep=1.5pt](3) at
(1,1) {};

    \node[draw,shape=circle,draw=black,minimum size=0.6pt,inner sep=1.5pt](4) at
(0,1) {};
     \node[draw,shape=circle,draw=black,minimum size=0.6pt,inner sep=1.5pt](5)
at (0.5,1.8) {};
 \node[] at (0.5,2.1) {$v'$};

 \node[] at (-0.3,1) {$v$};
      \node[] at (1.3,1) {$u$};
 \draw[-,line width=2pt] (4)--(5)--(3)--(4);
 \draw[-,line width=2pt] (4)--(1)--(2)--(4);
 
 \node at (0.5,-0.5) {\footnotesize (iv)};
 \node at (2,0) {};
 \end{tikzpicture}
   \begin{tikzpicture}
 \node[draw,shape=circle,draw=black,minimum size=0.6pt,inner sep=1.5pt](1) at
(0,0) {};
  \node[draw,shape=circle,draw=black,minimum size=0.6pt,inner sep=1.5pt](2) at
(1,0) {};
   \node[draw,shape=circle,draw=black,minimum size=0.6pt,inner sep=1.5pt](3) at
(1,1) {};
    \node[draw,shape=circle,draw=black,minimum size=0.6pt,inner sep=1.5pt](4) at
(0,1) {};
     \node[draw,shape=circle,draw=black,minimum size=0.6pt,inner sep=1.5pt](5)
at (0.5,0.5) {};
 \draw[-,line width=2pt] (1)--(5)--(3);
 \draw[-,line width=2pt] (4)--(1)--(2)--(3)--(4);
 \node[] at (-0.3, 1) {$t$};
  \node[] at (-0.3,0) {$u$};
      \node[] at (1.3,1) {$v$};
      
 \node at (0.5,-0.5) {\footnotesize (v)};
 \node at (2,0) {};
 \end{tikzpicture}
 \caption{\label{fig:graphe5} The five possibilities of $2$-degenerated graphs
on five vertices with six edges}
 \end{center}
 \end{figure}
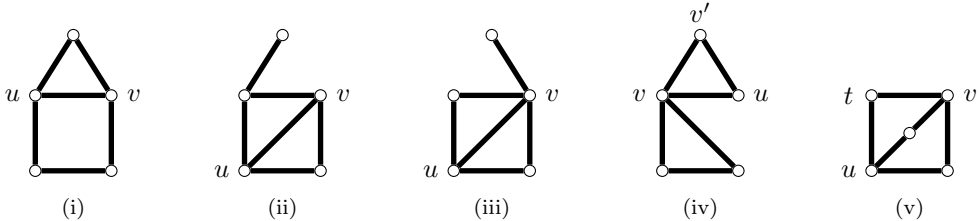

If $G''$ is a graph in part (i), (ii) or (iii) of
Figure~\ref{fig:graphe5}, then again one could repeat the arguments of
the proof of Theorem~\ref{EIDare2Degenerated} with $G'$ and the edge
$uv$ of the corresponding figure, to obtain a contradiction.

Suppose $G''$ is isomorphic to the graph of Figure~\ref{fig:graphe5}(iv).
Since $G''$ is not pendant-free, there must be at least one more vertex in $G'$.
Let $v_6$ be as in the sequence obtained by the 2-degeneracy of $G'$. Since $G'$
has exactly $2n-4$ edges, $v_6$ must have exactly two neighbours in $G''$. By
the symmetry of the four vertices of degree 2 in $G''$, we may assume $uv_6\in
\mathcal{C}_E$. Then $u$ and $v$ are both of degree at least 3 in $G'$.
Therefore, we could again repeat the argument of
Theorem~\ref{EIDare2Degenerated} with $G'$ and $uv$, where only one of  the
configurations of this theorem, namely \ref{fig:2degeneratedC2c}, matches
$G''$. Furthermore, if this happens then $v'v_6$ should also be an edge of
$G'$. Now $u$ and $v'$ are both of degree at least 3 and we apply the argument
of Theorem~\ref{EIDare2Degenerated} with $G'$ and $uv'$ to obtain a
contradiction.

Finally, let $G''$ be isomorphic to the graph of Figure~\ref{fig:graphe5}(v).
We claim that every other vertex $v_i$ ($i\geq 6$) is adjacent, in $G'$, only
to $u$ and $v$. By contradiction suppose $v_6$ is adjacent to $t$. Then
using the technique of Theorem~\ref{EIDare2Degenerated} applied on $G'$ and
$tu$ (respectively $tv$), we conclude that $v_6$ is adjacent to $u$
(respectively $v$).

Since $|E(G')|=|\mathcal{C}_{E}|=2n-4$, $G'$ is a spanning subgraph of $G$. But
then it is easy to verify that $\mathcal{C}_E\setminus\{xu,xv\}$ is an
edge-identifying code of $G$ --- a contradiction.
\end{proof}

We note that $\EID(K_{2,n})=2n-2=2|V(K_{2,n})|-6$ thus this bound cannot be
improved much.

Corollary~\ref{Bound2|V(G)|-5} implies that Conjecture~\ref{ConjFKKR} holds for a large subclass of line graphs:
\begin{corollary}
If $G$ is a pendant-free graph on $n$ vertices and with average degree 
${\bar d}(G)\geq 5$, then we have
$\ID(\mathcal L(G))\leq n-\frac{n}{\Delta(\mathcal L(G))}$. 
\end{corollary}

%We first note that $\sum_i=1^{|V(G)|} d_i= \bar d(G) |V(G)|$ and, therefore there is a
\begin{proof} 
Let $u$ be a vertex of degree $d(u)\geq \bar d(G)\geq 5$. Since $G$ is pendant-free there
is at least one neighbour $v$ of $u$ that is of degree at least~2. Thus there is an edge $uv$
in $G$ with $d(u)+d(v) \geq \bar d(G)+2$ and, therefore, $\Delta ( \mathcal L(G)) \geq \bar d(G)$.
Hence, considering Corollary~\ref{Bound2|V(G)|-5}, it is enough to show that 
$2|V(G)|-5  \leq |E(G)|-\frac{|E(G)|}{\bar d(G)}$. %A

To this end, since $\bar d(G) \geq 5$, we have $4 |V(G)|\leq (\bar d(G)-1)|V(G)|$, therefore,
$$4|V(G)|-10 \leq (\bar d(G)-1)|V(G)|.$$ Mutiplying both sides by $\frac{\bar
d(G)}{2}$ we have: %A
$$(2|V(G)|-5) \bar d(G) \leq (\bar d(G)-1)\frac{\bar d(G)}{2}|V(G)|=(\bar
d(G)-1) |E(G)|.$$
\end{proof}

\section{Complexity}\label{Complexity}

This section is devoted to the study of the decision problem
associated to the concept of edge-identifying codes. Let us first
define the decision problems we use.
The IDCODE problem is defined as follows:

~

\begin{minipage}{\textwidth}
\textbf{IDCODE}
%\noindent

INSTANCE: A graph $G$ and an integer~$k$.

%\noindent
QUESTION: Does $G$ have an identifying code of size at most~$k$?
\end{minipage}
\vspace{0.3cm}

IDCODE was proved to be NP-complete even when restricted to the class
of bipartite graphs of maximum degree~3 (see~\cite{CHL03}) or to the class of
planar graphs of maximum degree~4 and arbitrarily large girth (see \cite{A10}).
The EDGE-IDCODE problem is defined as follows:

~

\begin{minipage}{\textwidth}
%\noindent
\textbf{EDGE-IDCODE}

%\noindent
INSTANCE: A graph $G$ and an integer~$k$.

%\noindent
QUESTION: Does $G$ have an edge-identifying code of size at most~$k$?
\end{minipage}
\vspace{0.3cm}

We will prove that EDGE-IDCODE is NP-hard in some restricted class of
graphs by reduction from PLANAR ($\le 3,3$)-SAT, which is a variant of
the SAT problem and is defined as follows~\cite{DJPSY94}:

~

\begin{minipage}{\textwidth}
%\noindent
\textbf{PLANAR ($\le 3,3$)-SAT}

%\noindent
INSTANCE: A collection $\mathcal{Q}$ of clauses over a set $X$ of
boolean variables, where each clause contains at least two and at most
three distinct literals (a variable $x$ or its negation
$\overline{x}$). Moreover, each variable appears in exactly three
clauses: twice in its non-negated form, and once in its negated
form. Finally, the bipartite incidence graph of $\mathcal{Q}$, denoted
$B(\mathcal{Q})$, is planar ($B(\mathcal{Q})$ has vertex set
$\mathcal{Q}\cup X$ and $Q\in\mathcal{Q}$ is adjacent to $x\in
X$ if $x$ or $\overline{x}$ appears in clause $Q$).

%\noindent
QUESTION: Can $\mathcal{Q}$ be satisfied, i.e. is there a truth
assignment of the variables of $X$ such that each clause contains at
least one true literal?
\end{minipage}
\vspace{0.3cm}

PLANAR ($\le 3,3$)-SAT is known to be NP-complete~\cite{DJPSY94}. We
are now ready to prove the main result of this section.

\begin{theorem}~\label{thm:NPhard}
  EDGE-IDCODE is NP-complete even when restricted to bipartite planar
  graphs of maximum degree~3 and arbitrarily large girth.
\end{theorem}

\begin{proof}
The problem is clearly in NP: given a subset $C$ of edges of $G$,
one can check in polynomial time whether it is an edge-identifying
code of $G$ by computing the sets $C\cap N[e]$ for each edge $e$
and comparing them pairwise.

  We now reduce PLANAR ($\le 3,3$)-SAT to EDGE-IDCODE. We first give
  the proof for the case of girth 8 and show that it can be easily
  extended to an arbitrarily large girth.

  We first need to define a generic sub-gadget (denoted $P$-gadget)
  that will be needed for the reduction. In order to have more compact
  figures, we will use the representation of this gadget as drawn in
  Figure~\ref{fig:smallgadgets}.  We will say that a $P$-gadget is
  {\it attached} at some vertex $x$ if $x$ is incident to edge $a$ of
  the gadget as depicted in the figure. When speaking of a $P$-gadget
  as a subgraph of a graph $G$, we always mean that it forms an
  induced subgraph of $G$, that is, there are no other edges within
  the gadget than $\{a,b,c,d,e\}$ in
  Figure~\ref{fig:smallgadgets}. Moreover, vertex $x$ is the only
  vertex of the $P$-gadget which may be joined by an edge to other
  vertices outside the gadget.

\begin{figure}[!ht]
\centering
\begin{tikzpicture}[join=bevel,inner sep=0.5mm,line width=0.8pt, scale=0.4]
  \node at (-6.5,0) {};
  \path (2,0) node[draw,shape=circle,fill=black] (v1) {};
  \path (v1)+(0,-0.75) node {$x$};
  \path (2,2) node[draw,shape=circle,fill=black] (p0) {};  
  \path (p0)+(0.5,-1) node (f) {$a$};
  \path (0.6,3.4) node[draw,shape=circle,fill=black] (q0) {}; 
  \path (p0)+(-1.3,0.5) node (g) {$b$};

  \path (3.4,3.4) node[draw,shape=circle,fill=black] (r0) {}; 
  \path (p0)+(1.2,0.5) node (i) {$c$};
  \path (r0)+(0,2) node[draw,shape=circle,fill=black] (r1) {}; 
  \path (r0)+(0.5,1) node (j) {$d$};
  \path (r1)+(0,2) node[draw,shape=circle,fill=black] (r2) {}; 
  \path (r1)+(0.5,1) node (k) {$e$};

  \path (2,-3) node (g) {$G$};
 
  \draw [-] (v1) -- (p0) -- (q0)
            (p0) -- (r0) -- (r1) -- (r2);
  \draw (v1)+(0,-3) ellipse (4cm and 3cm);

   \path (8,1) node (arrow) {};
  \draw [->] (arrow) -- +(2,0);

  \path (16,0) node[draw,shape=circle,fill=black] (w1) {};
  \path (w1)+(0,-0.75) node {$x$};
  \path (16,-3) node {$G$};
  \path (16,4) node {$P$};

  \draw [-] (w1) .. controls +(-8,10) and +(8,10) .. (w1);
  \draw (w1)+(0,-3) ellipse (4cm and 3cm);
\end{tikzpicture}
\caption{The generic $P$-gadget}
\label{fig:smallgadgets}
\end{figure}
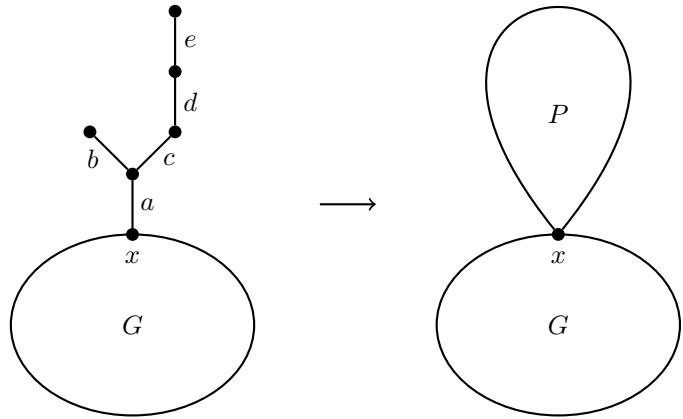

We make the following claims.

\begin{claim}\label{clm:Pgadg} In any graph containing a $P$-gadget, at least three edges of this
gadget must belong to any edge-identifying code.
\end{claim}

Claim~\ref{clm:Pgadg} is true because $d$ is the only edge separating
$b$ and $c$. Similarly $c$ is the only edge separating $d$ and $e$.
Finally, in order to separate $d$ and $c$, one has to take at least one of
$a$, $b$ or $e$.

\begin{claim}\label{clm:Fgadg2}
 If $G$ is a pendant-free graph obtained from a graph $H$ with a $P$-gadget
attached at a vertex $x$ of $H$, then any edge-identifying code of $G$ must contain
an edge of $H$ incident to $x$.
\end{claim}

Claim~\ref{clm:Fgadg2} follows from the fact that edge $a$ must be
separated from edge $b$.

\vspace{0.3cm}
We are now ready to describe the reduction.

Given an instance $\mathcal{Q}=\{Q_1,\ldots,Q_m\}$ of PLANAR ($\le 3,3$)-SAT
over the set of boolean variables $X=\{x_1,\ldots,x_n\}$ together with
an embedding of its bipartite incidence graph $B(\mathcal{Q})$ in the
plane, we build the graph $G_\mathcal{Q}$ as follows.

For each variable $x_j$ and clause $Q_i$ we build the subgraphs
$G_{x_j}$ and $G_{Q_i}$ respectively, as shown in Figure~\ref{Gadget}.
We recall that a given variable $x_j$ appears in positive form in
exactly two clauses, say $Q_p$, $Q_q$, and in negative form in exactly
one clause, say $Q_r$.  We then unify\symbolfootnote[1]{We use the
  term ``unify'' instead of the usual term ``identify'' in order to
  avoid confusion with identifying codes.} vertex $x_j^1$ of
$G_{x_{j}}$ with vertex $l_{p_{k}}$ of $G_{Q_{p}}$ which corresponds
to $x_j$ . We do a similar unification for vertices $x_j^2$ and
$\overline{x_j}^1$ with corresponding vertices from $G_{Q_{q}}$ and
$G_{Q_{r}}$. The intuition is that vertices of the form $l_{i_j}$ in
the clause gadgets will represent literals of the clauses, and
vertices of the form $x_i^j$, $\overline{x_i}^j$ of the vertex gadgets represent positive and negative occurences of a variable, respectively.

This can be done while ensuring the
planarity of $G_\mathcal{Q}$, using the given planar embedding
of $B(\mathcal{Q})$. Moreover, $G_\mathcal{Q}$ is bipartite because
$B(\mathcal{Q})$ is bipartite, there are no odd cycles in the
variable and clause gadgets and there is no path of odd length between
$l_{i_j}$'s. Finally, it is easy to see that $G_\mathcal{Q}$ has
maximum degree~3 and girth~8. Since a clause gadget has fourty-five vertices
and a variable gadget, fourty-two vertices, $G_\mathcal{Q}$ has $45m+42n$
vertices and, therefore, the construction has polynomial size in terms
of the size of $\mathcal{Q}$.

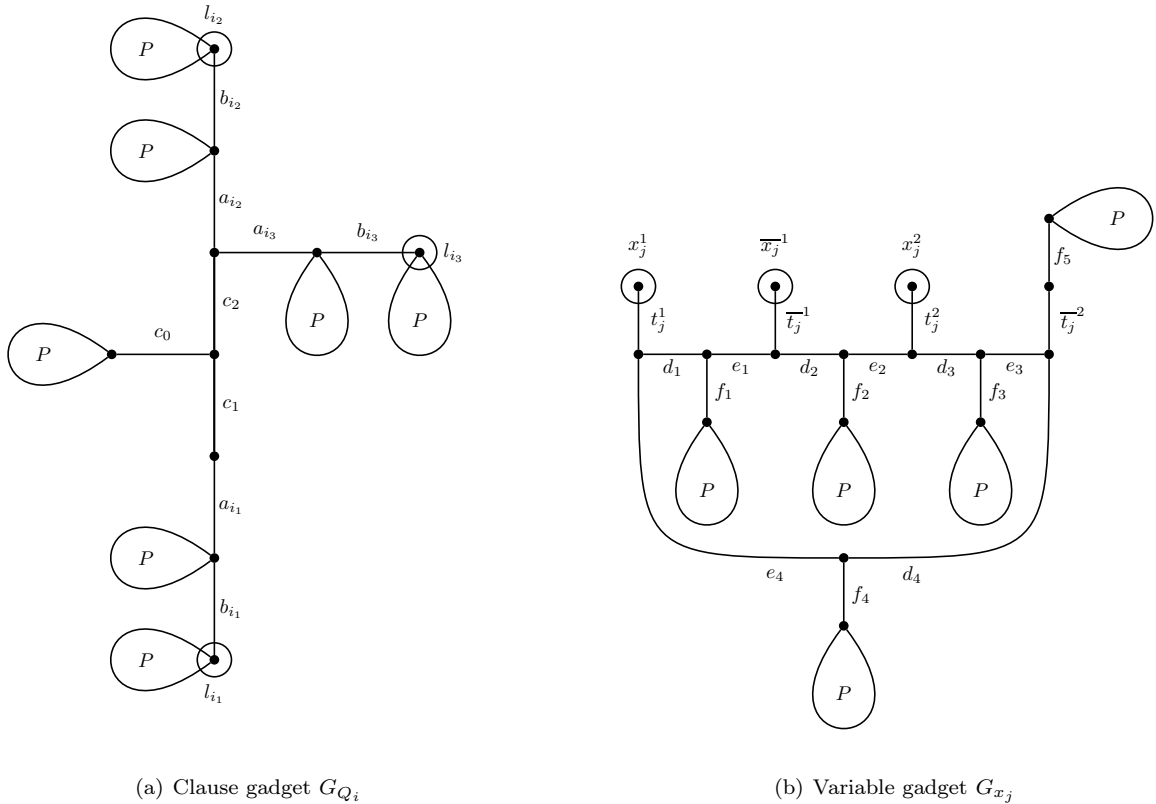
\begin{figure}[!ht]
\centering
\subfigure[Clause gadget $G_{Q_i}$]{
\scalebox{0.75}{\begin{tikzpicture}[join=bevel,inner sep=0.5mm,line width=0.8pt, scale=0.6, rotate=90]
  \path (0,0) node[draw,shape=circle,fill=black] (v0) {};
  \path (v0)+(-1,0) node {$l_{i_1}$};
  \path (v0)+(1.5,-0.5) node {$b_{i_1}$};
  \path (3,0) node[draw,shape=circle,fill=black] (v1) {};
  \path (v1)+(1.5,-0.5) node {$a_{i_1}$};
  \path (6,0) node[draw,shape=circle,fill=black] (v5) {};
  \path (v5)+(1.5,-0.5) node {$c_1$};
  \path (9,0) node[draw,shape=circle,fill=black] (v6) {};
  \path (v6)+(1.5,-0.5) node {$c_2$};
  \path (v6)+(0,3) node[draw,shape=circle,fill=black] (v) {};
  \path (v)+(0.6,-1.5) node {$c_0$};
  \path (12,0) node[draw,shape=circle,fill=black] (v2) {};
  \path (v2)+(1.5,-0.5) node {$a_{i_2}$};
  \path (15,0) node[draw,shape=circle,fill=black] (v3) {};
  \path (v3)+(1.5,-0.5) node {$b_{i_2}$};
  \path (18,0) node[draw,shape=circle,fill=black] (v4) {};
  \path (v4)+(1,0) node {$l_{i_2}$};
  \path (12,-3) node[draw,shape=circle,fill=black] (v7) {};
  \path (v7)+(0.6,1.5) node {$a_{i_3}$};
  \path (12,-6) node[draw,shape=circle,fill=black] (v8) {};
  \path (v8)+(0.6,1.5) node {$b_{i_3}$};
  \path (v8)+(0,-1) node {$l_{i_3}$};

  \draw [-] (v0) -- (v1) -- (v2) -- (v3) -- (v4)
            (v5) -- (v6) -- (v2) -- (v7) -- (v8)
            (v) -- (v6);

  \draw (v0) ellipse (0.5cm and 0.5cm)
        (v4) ellipse (0.5cm and 0.5cm)
        (v8) ellipse (0.5cm and 0.5cm);

  \path (v0)+(0,2) node  {$P$};
  \path (v1)+(0,2) node {$P$};
  \path (v)+(0,2) node {$P$};
  \path (v3)+(0,2) node  {$P$};
  \path (v4)+(0,2) node {$P$};
  \path (v7)+(-2,0) node  {$P$};
  \path (v8)+(-2,0) node  {$P$};

  \draw [-] (v0) .. controls +(-3,4) and +(3,4) .. (v0)
            (v1) .. controls +(-3,4) and +(3,4) .. (v1)
            (v3) .. controls +(-3,4) and +(3,4) .. (v3)
            (v4) .. controls +(-3,4) and +(3,4) .. (v4)
            (v7) .. controls +(-4,3) and +(-4,-3) .. (v7)
            (v8) .. controls +(-4,3) and +(-4,-3) .. (v8)
            (v) .. controls +(-3,4) and +(3,4) .. (v);
\end{tikzpicture}}
}
\qquad
\subfigure[Variable gadget $G_{x_j}$]{
\scalebox{0.75}{\begin{tikzpicture}[join=bevel,inner sep=0.5mm,line width=0.8pt, scale=0.6]
  \path (0,0) node[draw,shape=circle,fill=black] (v0) {};
  \path (v0)+(1,-0.4) node {$d_1$};
  \path (2,0) node[draw,shape=circle,fill=black] (v1) {};
  \path (v1)+(1,-0.4) node {$e_1$};
  \path (v1)+(0,-2) node[draw,shape=circle,fill=black] (v02) {};
  \path (v02)+(0.5,0.9) node {$f_1$};

  \path (4,0) node[draw,shape=circle,fill=black] (v2) {};
  \path (v2)+(1,-0.4) node {$d_2$};
  \path (6,0) node[draw,shape=circle,fill=black] (v3) {};
  \path (v3)+(1,-0.4) node {$e_2$};
  \path (v3)+(0,-2) node[draw,shape=circle,fill=black] (v24) {};
  \path (v24)+(0.5,0.9) node {$f_2$};
  \path (8,0) node[draw,shape=circle,fill=black] (v4) {};
  \path (v4)+(1,-0.4) node {$d_3$};
  \path (10,0) node[draw,shape=circle,fill=black] (v5) {};
  \path (v5)+(1,-0.4) node {$e_3$};
  \path (v5)+(0,-2) node[draw,shape=circle,fill=black] (v46) {};
  \path (v46)+(0.5,0.9) node {$f_3$};
  \path (12,0) node[draw,shape=circle,fill=black] (v6) {};
  \path (6,-6) node[draw,shape=circle,fill=black] (v7) {};
  \path (v7)+(-2,-0.5) node {$e_4$};
  \path (v7)+(0,-2) node[draw,shape=circle,fill=black] (v60) {};
  \path (v60)+(0.5,0.9) node {$f_4$};
  \path (v7)+(2,-0.5) node {$d_4$};

  \path (0,2) node[draw,shape=circle,fill=black] (u0) {};
  \path (u0)+(0.6,-1) node {$t_j^1$};
  \path (u0)+(0,1.2) node {$x_j^1$};
  \path (4,2) node[draw,shape=circle,fill=black] (u1) {};
  \path (u1)+(0.7,-1) node {$\overline{t_{j}}^1$};
  \path (u1)+(0,1.2) node {$\overline{x_j}^1$};
  \path (8,2) node[draw,shape=circle,fill=black] (u2) {};
  \path (u2)+(0.6,-1) node {$t_j^2$};
  \path (u2)+(0,1.2) node {$x_j^2$};
  \path (12,2) node[draw,shape=circle,fill=black] (u3) {};
  \path (u3)+(0.7,-1) node {$\overline{t_{j}}^2$};

  \path (12,4) node[draw,shape=circle,fill=black] (x) {};
  \path (x)+(0.4,-1.1) node {$f_5$};
  %\path (12,6) node[draw,shape=circle,fill=black] (y) {};

  \draw (u0) ellipse (0.5cm and 0.5cm)
        (u1) ellipse (0.5cm and 0.5cm)
        (u2) ellipse (0.5cm and 0.5cm);
 
  \draw [-] (u0) -- (v0) -- (v1) -- (v2) -- (v3) -- (v4) -- (v5) -- (v6) -- (u3) -- (x) %-- (y)
            (u1) -- (v2)
            (u2) -- (v4)
            (v1) -- (v02)
            (v3) -- (v24)
            (v5) -- (v46)
            (v7) -- (v60)
            (v0) .. controls +(0,-6) .. (v7)
            (v6) .. controls +(0,-6) .. (v7);

  \path (v02)+(0,-2) node  {$P$};
  \path (v24)+(0,-2) node {$P$};
  \path (v46)+(0,-2) node  {$P$};
  \path (v60)+(0,-2) node  {$P$};

  \path (14,4) node  {$P$};

  \draw [-] (v02) .. controls +(-3,-4) and +(3,-4) .. (v02)
            (v24) .. controls +(-3,-4) and +(3,-4) .. (v24)
            (v46) .. controls +(-3,-4) and +(3,-4) .. (v46)
            (v60) .. controls +(-3,-4) and +(3,-4) .. (v60)
            (x) .. controls +(4,-3) and +(4,3) .. (x);
\end{tikzpicture}}
}
\caption{Reduction gadgets for clause $Q_i$ and variable
  $x_j$}
\label{Gadget}
\end{figure}

We will need two additional claims in order to complete the proof.

\begin{claim}\label{clm:vargadg}
In a variable gadget $G_{x_j}$, in order to separate the four pairs of edges
$\{d_i,e_i\}$ for $1\leq i\leq 4$, at least two edges of
$A=\{d_i,e_i~|~1\leq i\leq 4\}\cup\{t_j^1, \overline{t_{j}}^1,
t_j^2, \overline{t_{j}}^2\}$ belong to any edge-identifying code
$C$. Moreover, if $|C\cap A|=2$, then either $C\cap A=\{t_j^1,
t_j^2\}$ or $C\cap A=\{\overline{t_{j}}^1, \overline{t_{j}}^2\}$.
\end{claim}

The first part of Claim~\ref{clm:vargadg} follows from the fact that
the two edges of each of the pairs $\{d_1,e_1\}$ and $\{d_3,e_3\}$
must be separated. The second follows from an easy case analysis.

The following claim follows directly from Claim~\ref{clm:Fgadg2}.

\begin{claim}\label{clm:Fpath}
Let $v_1v_2v_3v_4$ be a path of four vertices of $G_\mathcal{Q}$ where
each of the vertices $v_2$ and $v_3$ has its own $P$-gadget attached
and both $v_2$ and $v_3$ have degree~3. Then, at least one of the
three edges of the path belong to any identifying code of the
graph. If exactly one belongs to a code, it must be $v_2v_3$.
\end{claim}

We now claim that $\mathcal{Q}$ is satisfiable if and only if
$G_\mathcal{Q}$ has an edge-identifying code of size at most $k=25m+22n$.

\vspace{0.3cm}
For the sufficient side, given a truth assignment of the variables
satisfying $\mathcal{Q}$, we build an edge-identifying code $\mathcal C$ as  follows. For
each $P$-gadget, edges $a,c,d$ are in $\mathcal C$. For   each clause gadget $G_{Q_i}$, edge
$c_0$ belongs to $\mathcal C$. For each   literal $l_{i_k}$ of $Q_i$, $1\le k\le 3$, if $l_{i_k}$ is true,
edge $a_{i_k}$ belongs to $\mathcal C$; otherwise, edge $b_{i_k}$ belongs to $\mathcal C$. If $Q_i$ has only
two literals and vertex $l_{i_k}$ is the vertex not corresponding to a literal of $Q_i$, then
edge $b_{i_k}$ belongs to $\mathcal C$. Now, one can see that all edges of $G_{Q_i}$ are dominated.
Furthermore, all pairs of edges of $G_{Q_i}$ are separated. This can be easily
seen for all pairs besides $\{c_1, c_2\}$. For this pair, since we are
considering a satisfying assignment of $\mathcal{Q}$, in every clause $Q_i$ of
$\mathcal{Q}$, there exists a true literal. Hence,  for each clause $Q_i$, at least one edge
$a_{i_j}$ with $1\le j \le 3$, must be in the code and, therefore, the pair $\{c_1, c_2\}$ is
separated.

Next, in each variable gadget $G_{x_j}$, if $x_j$ is true, edges $t_j^1$ and $t_j^2$ belong
to $\mathcal C$. Otherwise, edges $\overline{t_j}^1$ and $\overline{t_j}^2$ belong to $\mathcal C$. Edges
$f_1,f_2,f_3,f_4$ and $f_5$ also belong to $\mathcal C$. Because of this choice, all edges of
 $G_{x_j}\setminus \{t_j^1,t_j^2,\overline{t_j}^1\}$ are dominated. Since each of the three
edges $t_j^1,t_j^2,\overline{t_j}^1$ is incident to a vertex of a $P$-gadget of some clause
gadget, they are also dominated. Moreover, all pairs of edges containing at least one edge of
$G_{x_j}\setminus \{t_j^1,t_j^2,\overline{t_j}^1\}$ are clearly separated. Now, since for
each $P$-gadget of the clause gadgets, edge $a$ is in $\mathcal C$, $t_j^1,t_j^2,\overline{t_j}^1$
are separated from all edges in $G_{\mathcal{Q}}$.

We conclude that $\mathcal C$ is an edge-identifying code of size~$k$.

\vspace{0.3cm}
For the necessary side, let $\mathcal{C}'$ be an edge-identifying code of
$G_\mathcal{Q}$ with $|\mathcal{C}'|\leq k$. It follows from
Claim~\ref{clm:Pgadg} that at least three edges of each of the seven
$P$-gadgets of a clause gadget $G_{Q_i}$ must belong to
$\mathcal{C}'$. Moreover, by Claim~\ref{clm:Fgadg2}, edge $c_0$ is forced to be
in any code. Finally, by Claim~\ref{clm:Fgadg2}, for each vertex
$l_{i_k}$ ($1\le k\le 3$) of $G_{Q_i}$, at least one of the edges $a_{i_k}$ and $b_{i_k}$ is
in $\mathcal{C}'$.

Note that this is a total of at least twenty-five edges per clause
gadget.

Similarly, it follows from Claim~\ref{clm:Pgadg} that in each variable
gadget $G_{x_j}$, at least fifteen edges of $\mathcal{C}'$ are contained in the
$P$-gadgets of $G_{x_j}$. Following
Claim~\ref{clm:Fgadg2}, all edges $f_i$ ($1\leq i\leq 5$) belong to
$\mathcal{C}'$. Note that this is a total of at least twenty edges in each variable
gadget. We have considered $25m+20n$ edges of $\mathcal{C}'$ so far. Hence $2n$
edges remain to be considered. It follows from Claim~\ref{clm:vargadg}
that for each variable gadget, at least two additional edges belong to
$\mathcal{C}'$ (in order to separate the pairs $\{d_i,e_i\}$, for $1\leq i\leq
4$). Therefore, since $|C'|\leq k$, in each variable gadget, exactly two of these edges
belong to $\mathcal{C}'$. Hence, following the second part of
Claim~\ref{clm:vargadg}, either $\{t_j^1, t_j^2\}$ or
$\{\overline{t_{j}}^1, \overline{t_{j}}^2\}$ is a subset of $\mathcal{C}'$.

Remark that we have now considered all $k=25m+22n$ edges of $\mathcal{C}'$.
Therefore, in each clause gadget
$G_{Q_i}$, \emph{exactly} one of the edges $a_{i_k}$ and $b_{i_k}$ of
$G_{Q_i}$ belongs to $\mathcal{C}'$.

We can now build the following truth assignment: for each variable
gadget, if $\{t_j^1, t_j^2\}$ is a subset of $\mathcal{C}'$, $x_j$ is set to
TRUE. Otherwise, $\{\overline{t_{j}}^1, \overline{t_{j}}^2\}$ is a
subset of $\mathcal{C}'$ and $x_j$ is set to FALSE. Let us prove that this
assignment satisfies $\mathcal{Q}$.

In each clause gadget $G_{Q_i}$, note that edges $c_1$ and $c_2$ must
be separated by $\mathcal{C}'$; this means that one edge $a_{i_k}$ from
$\{a_{i_1},a_{i_2},a_{i_3}\}$ belongs to $\mathcal{C}'$. Hence, as noted in the
previous paragraph, $b_{i_k}\notin \mathcal{C}'$ and by Claim~\ref{clm:Fpath},
in the path formed by edges $\{a_{i_k},b_{i_k},t_j^1\}$, $t_j^1$
belongs to the code (without loss of generality, we suppose that
$l_{i_k}=x_j$ and $t_j^1$ is the edge of $G_{x_j}$ incident to vertex
$l_{i_k}$ of $G_{Q_i}$). Therefore, in the constructed truth assignment,
literal $l_{i_k}$ has value TRUE, and the clause is
satisfied. Repeating this argument for each clause shows that the
formula is satisfied.

\vspace{0.3cm}
Now, it remains to show that similar arguments can be used to prove the
final statement of the theorem for larger girth. Consider some integers $\lambda \ge 1$
and $\mu \ge 2$. We build the graph $G_\mathcal{Q}(\lambda,\mu)$ using
modified variable gadgets $G_{x_j}(\mu)$ and modified clause gadgets
$G_{Q_i}(\lambda)$, which are depicted in
Figure~\ref{Gadget_biggirth}. The construction is the same as in the
previous proof and $G_\mathcal{Q}(\lambda,\mu)$ has
$(36\lambda+9)m+(30\mu-18)n$ vertices. We claim that the girth of
$G_\mathcal{Q}(\lambda,\mu)$ is now  at least
$\min\{4\mu,8(\lambda+1)\}$. Indeed, $G_{x_j}(\mu)$ has a cycle of
size exactly $4\mu$ and since the girth of $B(\mathcal{Q})$ is at
least~$4$, it follows that the minimum length of a cycle between some
clause gadgets (at least two) and some variable gadgets (at least two)
is at least $4(2\lambda+1)+2+2=8(\lambda+1)$.

Now, using a similar proof as the proof for girth~8, it can be shown
that $\mathcal{Q}$ is satisfiable if and only if
$G_\mathcal{Q}(\lambda,\mu)$ has an identifying code of size at most
$k=(21\lambda+4)m + (17\mu - 12)n$.
\begin{figure}[!ht]
\centering
\subfigure[Clause gadget $G_{Q_i}(\lambda)$]{
\scalebox{0.6}{\begin{tikzpicture}[join=bevel,inner sep=0.5mm,line width=0.8pt, scale=0.6]
  %left part:
  \path (-6,0) node[draw,shape=circle,fill=black] (v0) {};
%  \path (v0)+(-1,0) node {$l_{i_1}$};
%  \path (v0)+(1.5,-0.6) node {$b_{i_1}^\lambda$};
  \path (-3,0) node[draw,shape=circle,fill=black] (v1) {};
%  \path (v1)+(1.5,-0.6) node {$a_{i_1}^\lambda$};
  \path (0,0) node[draw,shape=circle,fill=black] (v11) {};
  \path (v11)+(1.5,0) node {$...$};
  \path (3,0) node[draw,shape=circle,fill=black] (v12) {};
%  \path (v12)+(1.5,-0.6) node {$b_{i_1}^1$};
  \path (6,0) node[draw,shape=circle,fill=black] (v13) {};
%  \path (v13)+(1.5,-0.6) node {$a_{i_1}^1$};
  \path (9,0) node[draw,shape=circle,fill=black] (vn1) {};
%  \path (vn1)+(1.5,-0.6) node {$c_1$};
  \path (12,0) node[draw,shape=circle,fill=black] (vn2) {};
%  \path (vn2)+(1.5,-0.6) node {$c_2$};
  \path (12,3) node[draw,shape=circle,fill=black] (v) {};
%  \path (v)+(0.6,-1.5) node {$c_0$};

  % right part:
  \path (15,0) node[draw,shape=circle,fill=black] (v2) {};
%  \path (v2)+(1.5,-0.6) node {$a_{i_2}^1$};
  \path (18,0) node[draw,shape=circle,fill=black] (v3) {};
%  \path (v3)+(1.5,-0.6) node {$b_{i_2}^1$};
  \path (21,0) node[draw,shape=circle,fill=black] (v4) {};
  \path (v4)+(1.5,0) node {$...$};
  \path (24,0) node[draw,shape=circle,fill=black] (v41) {};
%  \path (v41)+(1.5,-0.6) node {$a_{i_2}^\lambda$};
  \path (27,0) node[draw,shape=circle,fill=black] (v42) {};
%  \path (v42)+(1.5,-0.6) node {$b_{i_2}^\lambda$};
  \path (30,0) node[draw,shape=circle,fill=black] (v43) {};
%  \path (v43)+(1,0) node {$l_{i_2}$};

  % bottom part:
  \path (15,-3) node[draw,shape=circle,fill=black] (v7) {};
%  \path (v7)+(0.6,1.5) node {$a_{i_3}^1$};
  \path (15,-6) node[draw,shape=circle,fill=black] (v71) {};
%  \path (v71)+(0.6,1.5) node {$b_{i_3}^1$};
  \path (15,-9) node[draw,shape=circle,fill=black] (v72) {};
  \path (v72)+(0,1.5) node[rotate=90] {$...$};
  \path (15,-12) node[draw,shape=circle,fill=black] (v73) {};
%  \path (v73)+(0.6,1.5) node {$a_{i_3}^\lambda$};

  \path (15,-15) node[draw,shape=circle,fill=black] (v8) {};
%  \path (v8)+(0.6,1.5) node {$b_{i_3}^\lambda$};
%  \path (v8)+(0,-1) node {$l_{i_3}$};

  \draw [-] (v0) -- (v1) -- (v11)
            (v) -- (vn2)
            (v12) -- (v13) -- (vn1) -- (vn2) -- (v2) -- (v3) -- (v4)
            (v41) -- (v42) -- (v43)
            (v2) -- (v7) -- (v71)
            (v72) -- (v73) -- (v8);

  \draw (v0) ellipse (0.5cm and 0.5cm)
        (v43) ellipse (0.5cm and 0.5cm)
        (v8) ellipse (0.5cm and 0.5cm);

  \path (v0)+(0,2) node  {$P$};
  \path (v1)+(0,2) node {$P$};
  \path (v11)+(0,2) node  {$P$};
  \path (v12)+(0,2) node {$P$};
  \path (v13)+(0,2) node  {$P$};
  \path (v)+(0,2) node  {$P$};

  \path (18,2) node  {$P$};
  \path (21,2) node {$P$};
  \path (24,2) node  {$P$};
  \path (27,2) node {$P$};
  \path (30,2) node {$P$};

  \path (13,-3) node  {$P$};
  \path (13,-6) node  {$P$};
  \path (13,-9) node  {$P$};
  \path (13,-12) node  {$P$};
  \path (13,-15) node  {$P$};

  \draw [-] (v0) .. controls +(-3,4) and +(3,4) .. (v0)
            (v1) .. controls +(-3,4) and +(3,4) .. (v1)
            (v11) .. controls +(-3,4) and +(3,4) .. (v11)
            (v12) .. controls +(-3,4) and +(3,4) .. (v12)
            (v13) .. controls +(-3,4) and +(3,4) .. (v13)
            (v) .. controls +(-3,4) and +(3,4) .. (v)

            (v3) .. controls +(-3,4) and +(3,4) .. (v3)
            (v4) .. controls +(-3,4) and +(3,4) .. (v4)
            (v41) .. controls +(-3,4) and +(3,4) .. (v41)
            (v42) .. controls +(-3,4) and +(3,4) .. (v42)
            (v43) .. controls +(-3,4) and +(3,4) .. (v43)

            (v7) .. controls +(-4,3) and +(-4,-3) .. (v7)
            (v71) .. controls +(-4,3) and +(-4,-3) .. (v71)
            (v72) .. controls +(-4,3) and +(-4,-3) .. (v72)
            (v73) .. controls +(-4,3) and +(-4,-3) .. (v73)
            (v8) .. controls +(-4,3) and +(-4,-3) .. (v8);

\path (v0)+(-1,3.5) node (x1) {};
\path (v13)+(1,3.5) node (x2) {};
\path (0,5) node {$2\lambda$ times};
\draw[snake={brace},segment amplitude=5mm,line width=0.8pt] (x1) to (x2);

\path (v3)+(-1,3.5) node (x3) {};
\path (v43)+(1,3.5) node (x4) {};
\path (24,5) node {$2\lambda$ times};
\draw[snake={brace},segment amplitude=5mm,line width=0.8pt] (x3) to (x4);

\path (v7)+(-3.5,1) node (x5) {};
\path (v8)+(-3.5,-1) node (x6) {};
\path (9.5,-9) node[rotate=90] {$2\lambda$ times};
\draw[snake={brace},segment amplitude=5mm,mirror snake,line width=0.8pt] (x5) to (x6);
\end{tikzpicture}}}
\qquad
\subfigure[Variable gadget $G_{x_j}(\mu)$]{
\scalebox{0.6}{\begin{tikzpicture}[join=bevel,inner sep=0.5mm,line width=0.8pt, scale=0.6]
  \path (0,0) node[draw,shape=circle,fill=black] (v0) {};
%  \path (v0)+(1,-0.4) node {$d_1$};
  \path (2,0) node[draw,shape=circle,fill=black] (v1) {};
%  \path (v1)+(1,-0.4) node {$e_1$};
  \path (v1)+(0,-2) node[draw,shape=circle,fill=black] (v02) {};
%  \path (v1)+(0.5,-1.2) node {$f_1$};

  \path (4,0) node[draw,shape=circle,fill=black] (v2) {};
%  \path (v2)+(1,-0.4) node {$d_2$};
  \path (6,0) node[draw,shape=circle,fill=black] (v3) {};
%  \path (v3)+(1,-0.4) node {$e_2$};
  \path (v3)+(0,-2) node[draw,shape=circle,fill=black] (v24) {};
%  \path (v3)+(0.5,-1.2) node {$f_2$};

  \path (8,0) node[draw,shape=circle,fill=black] (v4) {};
%  \path (v4)+(1,-0.4) node {$d_3$};
  \path (10,0) node[draw,shape=circle,fill=black] (v5) {};
%  \path (v5)+(1,-0.4) node {$e_3$};
  \path (v5)+(0,-2) node[draw,shape=circle,fill=black] (v46) {};
%  \path (v5)+(0.5,-1.2) node {$f_3$};

  \path (12,0) node[draw,shape=circle,fill=black] (v6) {};
%  \path (v6)+(1,-0.4) node {$d_4$};
  \path (14,0) node[draw,shape=circle,fill=black] (v7) {};
  \path (v7)+(1.5,0) node {$...$};
  \path (v7)+(0,-2) node[draw,shape=circle,fill=black] (v68) {};
%  \path (v7)+(0.5,-1.2) node {$f_4$};

  \path (17,0) node[draw,shape=circle,fill=black] (v8) {};
  \path (v8)+(0,-2) node[draw,shape=circle,fill=black] (v810) {};
%  \path (v8)+(0.8,-1.2) node {$f_{2\mu-2}$};
%  \path (v8)+(1.1,-0.4) node {$e_{2\mu-2}$};
  \path (19,0) node[draw,shape=circle,fill=black] (v9) {};
%  \path (v9)+(1,-0.4) node {$d_{2\mu-1}$};

  \path (21,0) node[draw,shape=circle,fill=black] (v10) {};
%  \path (v10)+(1,-0.4) node {$e_{2\mu-1}$};
  \path (v10)+(0,-2) node[draw,shape=circle,fill=black] (v1012) {};
%  \path (v10)+(0.8,-1.2) node {$f_{2\mu-1}$};
  \path (23,0) node[draw,shape=circle,fill=black] (v11) {};

  \path (11.5,-7) node[draw,shape=circle,fill=black] (v12) {};
%  \path (v12)+(-6,-0.3) node {$e_{2\mu}$};
%  \path (v12)+(6,-0.3) node {$d_{2\mu}$};

  \path (v12)+(0,-2) node[draw,shape=circle,fill=black] (v120) {};
%  \path (v12)+(0.5,-1.2) node {$f_{2\mu}$};

  \path (0,2) node[draw,shape=circle,fill=black] (u0) {};
%  \path (u0)+(0.6,-1) node {$t_j^1$};
  \path (4,2) node[draw,shape=circle,fill=black] (u1) {};
%  \path (u1)+(0.7,-1) node {$\overline{t_{j}}^1$};
  \path (8,2) node[draw,shape=circle,fill=black] (u2) {};
%  \path (u2)+(0.6,-1) node {$t_j^2$};
  \path (12,2) node[draw,shape=circle,fill=black] (u3) {};
%  \path (u3)+(0.7,-1) node {$\overline{t_{j}}^2$};
  \path (19,2) node[draw,shape=circle,fill=black] (u4) {};
%  \path (u4)+(0.6,-1) node {$x_{j}^\mu$};
  \path (23,2) node[draw,shape=circle,fill=black] (u5) {};
%  \path (u5)+(0.7,-1) node {$\overline{t_{j}}^\mu$};

  \path (12,4) node[draw,shape=circle,fill=black] (x) {};
%  \path (x)+(0.5,-1.2) node {$f'_4$};
  \path (19,4) node[draw,shape=circle,fill=black] (y) {};
%  \path (y)+(0.8,-1.2) node {$f'_{2\mu-1}$};
  \path (23,4) node[draw,shape=circle,fill=black] (z) {};
%  \path (z)+(0.55,-1.2) node {$f'_{2\mu}$};

  \draw (u0) ellipse (0.5cm and 0.5cm)
        (u1) ellipse (0.5cm and 0.5cm)
        (u2) ellipse (0.5cm and 0.5cm);
 
  \draw [-] (u0) -- (v0) -- (v1) -- (v2) -- (v3) -- (v4) -- (v5) -- (v6) -- (u3) -- (x)
            (v1) -- (v02)
            (v3) -- (v24)
            (v5) -- (v46)
            (v7) -- (v68)
            (v8) -- (v810)
            (v10) -- (v1012)
            (v12) -- (v120)
            (u1) -- (v2)
            (u2) -- (v4)
            (v6) -- (v7)
            (v8) -- (v9) -- (u4) -- (y)
            (v9) -- (v10) -- (v11) -- (u5) -- (z)
            (v0) .. controls +(0.5,-6.5) .. (v12)
            (v11) .. controls +(-0.5,-6.5) .. (v12);

  \path (v02)+(0,-2) node  {$P$};
  \path (v24)+(0,-2) node {$P$};
  \path (v46)+(0,-2) node  {$P$};
  \path (v68)+(0,-2) node  {$P$};
  \path (v810)+(0,-2) node  {$P$};
  \path (v1012)+(0,-2) node  {$P$};
  \path (v120)+(0,-2) node  {$P$};

  \path (x)+(2,0) node  {$P$};
  \path (y)+(2,0) node  {$P$};
  \path (z)+(2,0) node  {$P$};

  \draw [-] (v02) .. controls +(-3,-4) and +(3,-4) .. (v02)
            (v24) .. controls +(-3,-4) and +(3,-4) .. (v24)
            (v46) .. controls +(-3,-4) and +(3,-4) .. (v46)
            (v68) .. controls +(-3,-4) and +(3,-4) .. (v68)
            (v810) .. controls +(-3,-4) and +(3,-4) .. (v810)
            (v1012) .. controls +(-3,-4) and +(3,-4) .. (v1012)
            (v120) .. controls +(-3,-4) and +(3,-4) .. (v120)
            (x) .. controls +(4,-3) and +(4,3) .. (x)
            (y) .. controls +(4,-3) and +(4,3) .. (y)
            (z) .. controls +(4,-3) and +(4,3) .. (z);

\path (x)+(-0.5,1.3) node (s) {};
\path (z)+(3,1.3) node (t) {};
\path (19,6.7) node {$2\mu-3$ times};
\draw[snake={brace},segment amplitude=5mm,line width=0.8pt] (s) to (t);

\end{tikzpicture}}
}
\caption{Reduction gadgets for clause $Q_i$ and variable
  $x_j$ for arbitrarily large girth}
\label{Gadget_biggirth}
\end{figure}
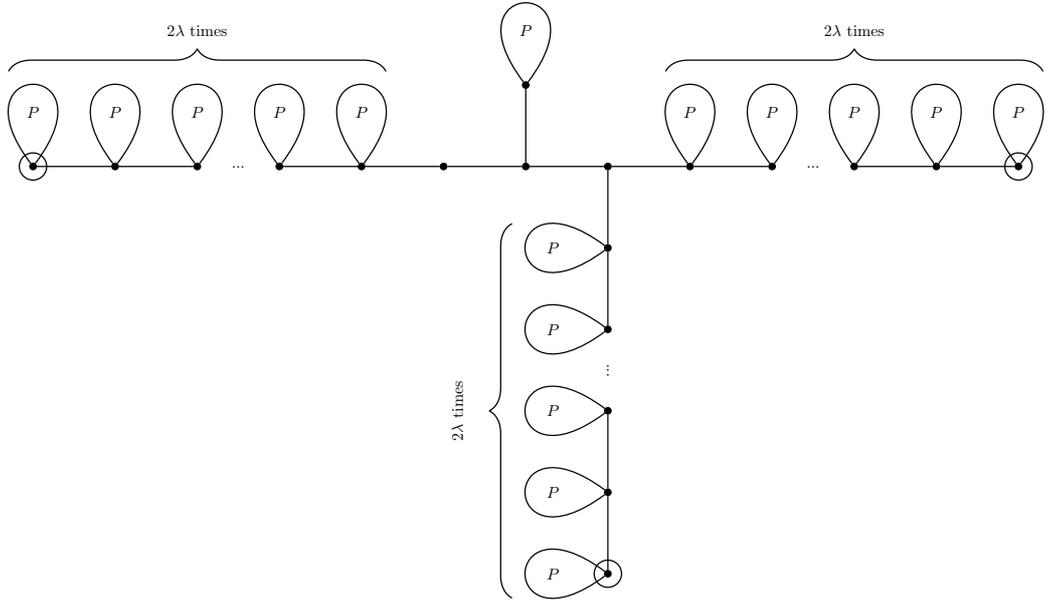
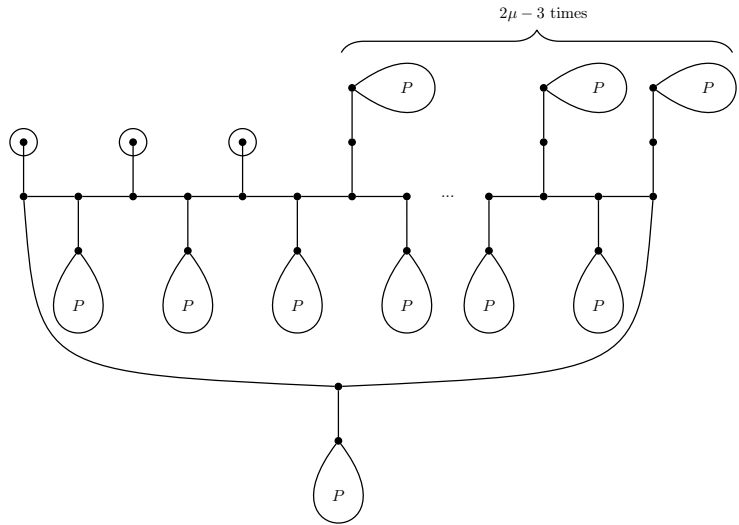
\end{proof}

Recall that a graph is \emph{perfect} if and only if for each of its
induced subgraphs $H$, the chromatic number of $H$ equals the clique
number of $H$.  It is known that a line graph $\mathcal{L}(G)$ is
perfect if and only if $G$ has no odd cycles of length more
than~3, see~\cite{T77}. Moreover, one can check that the line graphs
of the graphs constructed in the previous proof are planar, have
maximum degree~4 and clique number~3. Therefore, the following
corollary follows:

\begin{corollary}
IDCODE is NP-complete even when restricted to perfect 3-colorable
planar line graphs of maximum degree~4.
  %% IDCODE is NP-complete even when restricted to planar maximum
  %% degree~4 line graphs of bipartite planar graphs of maximum degree~3
  %% and arbitrarily large girth (and therefore to perfect planar line
  %% graphs of maximum degree~4, and to perfect 3-colourable line graphs
  %% of maximum degree~4).
\end{corollary}

Note that by Theorem~\ref{prop:low2} and
Corollary~\ref{Bound2|V(G)|-5}, we have $\tfrac{|V(G)|}{2}\leq
\EID(G)\leq |\mathcal{C}_E|\leq 2|V(G)|-3$ for any pendant-free graph
$G$ and any inclusionwise minimal edge-identifying code
$\mathcal{C}_E$ of $G$. Since one can construct such a code in
polynomial time, this gives a polynomial-time 4-approximation
algorithm for the optimization problem associated to EDGE-IDCODE:

\begin{theorem}
The optimization problems associated to EDGE-IDCODE in general graphs
and to IDCODE when restricted to line graphs are 4-approximable in
polynomial-time.
\end{theorem}

We remark that it is NP-hard to approximate the optimization version
of IDCODE within a factor of $o(\log(n))$ in general graphs on $n$
vertices (see~\cite{BLT06,S07}).

% We also note that the reduction of
% Theorem~\ref{thm:NPhard} can be used to prove that there is a constant
% $c$ such that it is NP-hard to $c$-approximate the optimization
% problem associated to EDGE-IDCODE~\cite{FloThesis}.

%% Note that besides this hardness result, Theorem~\ref{prop:low2} and
%% Corollary~\ref{Bound2|V(G)|-5} imply that there is a polynomial
%% 4-approximation algorithm of the optimization problem associated to
%% EDGE-IDCODE (by taking any minimal edge-identifying code). Hence, the
%% optimization problem associated to IDCODE is $4$-approximable in the
%% class of line graphs, whereas there is no constant-factor
%% approximation algorithm for IDCODE in general graphs
%% (see~\cite{BLT06}).

In the following, by slightly restricting the class of graphs
considered in Theorem~\ref{thm:NPhard}, we show that EDGE-IDCODE
becomes linear-time solvable in this restricted class. 

Let us first introduce some necessary concepts.

A graph property $\mathcal{P}$ is expressable in \emph{counting
  monadic second-order logic}, CMSOL for short (see~\cite{C90} for
further reference), if $\mathcal{P}$ can be defined using:
\begin{itemize}
\item vertices, edges, sets of vertices and sets of edges of a graph
\item the binary adjacency relation $adj$ where $adj(u,v)$ holds if and only if $u,v$ are two adjacent vertices
\item the binary incidence relation $inc$, where $inc(v,e)$ holds if and only if edge $e$ is incident to vertex $v$
\item the equality operator $=$ for vertices and edges
\item the membership relation $\in$, to check whether an element belongs to a set
\item the unary cardinality operator $card$ for sets of vertices
\item the logical operators OR, AND, NOT (denoted by $\vee$, $\wedge$, $\neg$)
\item the logical quantifiers $\exists$ and $\forall$ over vertices, edges, sets of vertices or sets of edges
\end{itemize}

%The $\tau_2$-monadic second-order logic (MSOL($\tau_2$) for short) is
%an extension of MSOL($\tau_1$). Moreover, MSOL($\tau_2$) allows
%the use of quantifiers $\exists$ and $\forall$ over edges and sets of
%edges of $G$.

It has been shown that CMSOL is particularly useful when combined with
the concept of the graph parameter \emph{tree-width} (we refer the
reader to~\cite{C90} for a definition). Some important classes of
graphs have bounded tree-width. For example, trees have tree-width at
most~1, series-parallel graphs have tree-width at most~2 and
outerplanar graphs have tree-width at most~3.

The following result shows that many graph properties can be checked in linear time for graphs of bounded tree-width.

\begin{theorem}[\cite{C90}]\label{thm:tw-msol}
Let $\mathcal{P}$ be a graph property expressable in
CMSOL and let $c$ be a constant. Then, for any graph $G$ of tree-width
at most $c$, it can be checked in linear time whether $G$ has property
$\mathcal{P}$.
\end{theorem}

We now show that CMSOL can be used in the context of edge-identifying codes:

\begin{proposition}\label{prop:msol}
Given a graph $G$ and an integer $k$, let $\mathcal{EID}(G,k)$ be the property
that $\EID(G)\leq k$. Property $\mathcal{EID}(G,k)$ can be expressed in CMSOL.
\end{proposition}
\begin{proof}
Let $V=V(G)$ and $E=E(G)$. We define the CMSOL relation
$dom(e,f)$ which holds if and only if $e,f$ are edges of $E$ and $e,f$
dominate each other, i.e. $e$ and $f$ are incident to the same
vertex. We have $dom(e,f):=\exists x\in V,(inc(x,e)\wedge inc(x,f))$.

Now we define $\mathcal{EID}(G,k)$ as follows:
\begin{eqnarray*}
\mathcal{EID}(G,k)&:=&\exists C, C\subseteq E, card(C)\leq k, \big(\forall e\in E, \exists f\in C, dom(e,f)\big) \wedge\\
&&\Big(\forall e\in E, \forall f\in E, e\neq f, \exists g\in C, \big((dom(e,g)
\wedge \neg dom(f,g)) \vee (dom(f,g) \wedge \neg dom(e,g))\big) \Big).
\end{eqnarray*}

\end{proof}

This together with Theorem~\ref{thm:tw-msol} implies the following corollary.

\begin{corollary}\label{cor:bTW}
EDGE-IDCODE can be solved in linear time for all classes of graphs
having their tree-width bounded by a constant.
\end{corollary}

This result implies, in particular, that one can find the
edge-identifying code number of a tree in linear time. Note that a
similar approach has been used in~\cite{M05} to show that this holds
for IDCODE as well.  

The proof of Theorem~\ref{thm:tw-msol} is
constructive and gives a linear-time algorithm, but it is very
technical and hides a large constant depending on the size of the CMSOL 
expression. Therefore, it would be interesting to give a simpler and more practical
linear-time algorithm for EDGE-IDCODE in trees. Observe that this has
been done in~\cite{A10} for the case of vertex-identifying codes.

\end{document}